\documentclass[11pt]{amsart}

\usepackage{amsmath,amssymb,amsthm}
\usepackage[margin=1.2in]{geometry}
\usepackage{graphicx}
\usepackage[colorlinks=true, linkcolor=blue, citecolor=blue]{hyperref}

\newtheorem{theorem}{Theorem}[section]
\newtheorem{proposition}[theorem]{Proposition}
\newtheorem{lemma}[theorem]{Lemma}
\newtheorem{corollary}[theorem]{Corollary}
\theoremstyle{definition}
\newtheorem{definition}[theorem]{Definition}
\newtheorem{assumption}[theorem]{Assumption}
\newtheorem{remark}[theorem]{Remark}

\newtheorem{thmx}{Theorem}

\newcommand{\Ld}{\mathcal{L}^d}
\newcommand{\LL}[1]{\mathcal{L}^{#1}}
\newcommand{\Hd}{\mathcal{H}^{d-1}}
\newcommand{\tr}{\operatorname{tr}}
\newcommand{\dive}{\nabla\!\cdot}
\newcommand{\R}{\mathbb{R}}
\newcommand{\T}{\mathbb{T}}
\newcommand{\PP}{\mathcal{P}}
\newcommand{\dd}{\,d}

\title[No-flux continuity equations and confined Lagrangian flows]{REFLECTED DIFFUSIONS,\\
NO-FLUX CONTINUITY EQUATIONS AND\\
CONFINED LAGRANGIAN FLOWS IN BOUNDED DOMAINS}

\author{Rama CONT}
\address{Mathematical Institute, University of Oxford}

\begin{document}
\maketitle
\begin{abstract}
Motivated by marginal distribution flows of reflected diffusions in bounded domains, we investigate when a density–flux pair solving a no-flux continuity equation admits a regular Lagrangian flow that remains in the closed domain and generates the prescribed density flow. Our first result gives sufficient conditions in terms of interior bounded-variation regularity, bounded-variation control on a boundary collar, a one-sided bound on an absolutely continuous divergence, and vanishing normal trace of the velocity. The proof uses the fact that tangency removes the singular boundary contribution to the divergence of the zero extension, thereby making the extended velocity admissible for the Ambrosio–DiPerna–Lions theory. We also establish two uniqueness results for no-flux Fokker–Planck equations: a duality result for bounded measurable drifts and a weighted-energy result for entrance-type drifts singular at the boundary.

Our second result shows that these boundary assumptions cannot be jointly relaxed so as to admit a boundary-current mechanism. We construct an explicit smooth density–flux pair carrying a boundary current—a tangential mass current of nonvanishing linear density along a wall where the volume density vanishes. Its density evolution is unique in a weighted class, and its characteristics are unique, confined, and transport the marginals, yet it admits no regular Lagrangian flow because the compressibility bound fails arbitrarily close to the initial time. Rigidity results delimit such failures and show that the relevant boundary hypotheses are structurally entangled. As an application, we provide precise regularity assumptions under which the reflection-free probability-flow ODE describes the flow of marginals of  reflected diffusion models,  after early stopping.  Our results 
provide a rigorous mathematical justification for using the  ODE-based sampling of reflected diffusion models under minimal regularity assumptions on the coefficients, and also indicate when such ODE-based samplers may fail.
\end{abstract}
Fokker–Planck equations; confined Lagrangian flows; generative models\\
DiPerna–Lions theory;  Continuity equation;  Linear transport equations.
49J52
 35J60 35Kxx  28A25

\newpage
\tableofcontents
\newpage
%=====================================================================
\section{Generative models, reflected diffusions and density flows in domains}\label{sec:intro}
%=====================================================================

\subsection{Generative modelling with reflected diffusions}\label{sec:generative}

{\it Reflected diffusion models} for constrained generative modelling were introduced  by Lou and Ermon \cite{LouErmon} and Fishman et al. \cite{fishman2024} to address the problem that standard
score-based diffusion models \cite{Song2021} can generate unnatural samples outside the support of the target data distribution. In this setting, the forward process is constrained to remain in a bounded domain $\Omega\subset \mathbb{R}^d$ representing the data constraints: it evolves as a {\it reflected diffusion process} in  $\Omega$ , starting from the data distribution $\mu_0$, supported on $\overline{\Omega}$:
\begin{equation}
dX_t =
f(X_t,t)\,dt + \sigma(t)\,dW_t - \nu(X_t)\,dL_t,
\qquad X_t\in\overline{\Omega},\qquad X_0\sim \mu_0\in {\mathcal P}(\overline{\Omega})\label{eq.RSDE}    
\end{equation}
where $\sigma:[0,T]\to\mathbb{R}_+ $ is a scalar function, $L_t$ is the  local time of $X$ at the boundary $\partial\Omega$ of the domain, which reflects the process into the interior whenever it reaches $\partial\Omega$, and $\nu(x)$ denotes the outward unit normal vector to the boundary at $x\in \partial\Omega$. This forward diffusion transports the data distribution $\mu_0$ (which may have a singular support) towards a (tractable) target distribution $\mu_T\in {\mathcal P}(\overline{\Omega})$ with density $p_T$.

Sampling is then done by {\it reversing} the density flow $(p_t)_{t\in[0,T]}$ of \eqref{eq.RSDE}. This can be done by simulating the time-reversed  reflected SDE \cite{Cattiaux88,petit1997}
\begin{equation}
\mathrm{d}Y_t = -\big(f(Y_t,T-t) - \sigma(T-t)^2 \nabla_x \log p_{T-t}(Y_t)\big)\,\mathrm{d}t + \sigma(T-t)\,\mathrm{d}\overline{B}_t -\nu(Y_t)\mathrm{d}\overline{L}_t,\quad
Y_t\in\overline{\Omega},\quad Y_0\sim p_T
\nonumber
\end{equation}
involving the {\it score function} $\nabla_x \log p_t$ of $X$, which is learned via   score matching \cite{LouErmon,fishman2024,Song2021,Song2021mle}. Then the marginal density   $Y_t$ is $q_t=p_{T-t}$ so $Y_T\sim \mu_0$ gives a sample from the data distribution.

Alternatively, as noted by Lou \& Ermon \cite{LouErmon}, the {\it same} probability flow $(p_t)_{t\in [0,T]}$ may be obtained by simulating a {\it deterministic}  probability flow ODE (PF-ODE)
\begin{equation}
\dot Z_t=f(Z_t,t)- \frac{\sigma(t)^2}{2}
\nabla \log p_t(Z_t),\qquad Z_0\sim \mu_0    \label{eq.PFODE} 
\end{equation}
driven by the velocity field $v(x,t)=f(x,t)-\tfrac 12\sigma^2(t)\nabla \log p_t(x)$. Simulating this ODE backwards in time starting from $Z_T\sim p_T$ then yields the backward flow $p_{T-t}$,
which enables fast deterministic sampling  of $\mu_0$ and exact likelihood evaluation. 
These methods yield state-of-the-art  generative models, 
achieving competitive performance on benchmark data sets such as CIFAR-10 and ImageNet \cite{LouErmon}. At the same time, their theoretical foundations remain to be fully explored \cite{chen2024}.

The claim \cite{LouErmon} that the marginal flows of the ODE \eqref{eq.PFODE} and the reflected SDE \eqref{eq.RSDE} have the same density flow $(p_t)_{t\in[0,T]}$ appears somewhat surprising at first glance.  Interestingly, as noted by 
Lou and Ermon \cite{LouErmon}, the  probability flow ODE \eqref{eq.PFODE} is the {\it same} as in the unconstrained case where $\Omega=\mathbb{R}^d$ \cite{pfode,Song2021} and, unlike the SDE \eqref{eq.RSDE}, does not contain any reflection term.  The density flow associated with the reflected SDE \eqref{eq.RSDE} satisfies a parabolic Fokker-Planck equation in a bounded domain with a no-flux boundary condition, while the flow associated with the ODE \eqref{eq.PFODE} is a {\it first-order}   transport equation. 
In fact,  if $p_t>0$ is regular enough, using the identity $\Delta p_t=\nabla\cdot{}(p_t\nabla\log p_t)$ the continuity equation for  the vector field $v(x,t)=f(x,t)-\tfrac 12\sigma^2(t)\nabla \log p_t(x)$
$$ \partial_t p + \dive(p\ v)=\partial_t p + \dive(p\ f)-\frac{\sigma(t)^2}{2}\Delta p_t$$
formally coincides with the Fokker-Planck equation for \eqref{eq.RSDE}! Hence, if uniqueness of solutions for both equations is satisfied, one   therefore expects the reflected diffusion and the PF-ODE to possess the same flow of marginal distributions $(p_t)_{t\in[0,T)}$ and \eqref{eq.PFODE} may indeed be used to sample from $p_t$ for $t>0$.

Needless to say, the above assertions assume  many  regularity, existence, uniqueness, and positivity properties which deserve to be qualified.
Our goal in this study is to give mathematically precise conditions for the above assertions to hold in a general setting. 

Note in particular that the ODE \eqref{eq.PFODE} itself carries no  boundary condition. Rather, the associated flux $J(x,t)=p_t(x)v(x,t)$ satisfies
$p_t(x)v(x,t)\cdot{}\nu(x)=0$ on the boundary.
One of the mathematical questions we address  is whether this 'Eulerian' condition actually produces a  regular Lagrangian flow confined to $\overline{\Omega}$. We will provide conditions for this to hold.
Alternatively, we will investigate how deterministic transport can fail when the density degenerates at the boundary.

\subsection{Probability flows and regular Lagrangian flows in bounded domains}\ \\

Motivated by these considerations, we investigate the following question:\\
\begin{center}
When does a probability flow over a bounded domain \\admit a deterministic transport representation?    
\end{center}
More precisely, given a probability flow $(\mu_t)_{t\in[0,T]}$ on $\overline{\Omega}$ with possibly singular initial law $\mu_0$,
 we seek to construct a vector field $v$ satisfying 'minimal' regularity  conditions, which
generates a unique regular Lagrangian flow  satisfying a no-flux condition at the boundary and transports $\mu_0$ along the prescribed flow.

We address these questions in a  general setting, for general density--flux pairs solving a no-flux continuity equation. We identify a BV normal-trace criterion that yields a confined regular Lagrangian flow through the Ambrosio--DiPerna--Lions theory, apply it to early-stopped reflected diffusions, and construct a boundary-current example in which invariant characteristics and exact marginal transport persist while regular-Lagrangian compressibility fails.

Let $\Omega \subset \R^d$ be a bounded Lipschitz domain with outer
unit normal $\nu$. The object of study is a
density--flux pair $(p, J)$ on $\bar\Omega\times[0,T]$, say $p$
and $J$ continuous on $\bar\Omega$ and continuously differentiable
in $\Omega$, satisfying the continuity equation together with a
\emph{no-flux} boundary condition:
\begin{equation}\label{eq:classical}
\left\{
\begin{aligned}
&\partial_t p(x,t) + \dive J(x,t) = 0,
&& (x,t)\in\Omega\times(0,T),\\[2pt]
&J(x,t)\cdot\nu(x) = 0,
&& (x,t)\in\partial\Omega\times(0,T),\\[2pt]
&p(\cdot,0) = p_0
&& \text{in } \Omega .
\end{aligned}
\right.
\end{equation}
The boundary condition $J(x,t)\cdot\nu(x) = 0$ means that no mass flows across the boundary
$\partial\Omega$: formally,
$$\tfrac{d}{dt}\int_\Omega p_t
= -\int_{\partial\Omega} J\cdot\nu\;d\Hd = 0$$
so total mass in $\Omega$ is conserved and the evolution is confined to
$\bar\Omega$. When the flux has the Fokker--Planck form
$J = fp - \tfrac12\sigma^2\nabla p$, the second line of
\eqref{eq:classical} is the co-normal condition
$\big(\tfrac12\sigma^2\nabla p - fp\big)\cdot\nu = 0$, a Robin-type
relation $\partial_\nu p = (2f\cdot\nu/\sigma^2)\,p$ which reduces
to the homogeneous Neumann condition only where $f\cdot\nu = 0$.
The classical formulation \eqref{eq:classical}, however,
presupposes that $J$ has a well-defined normal trace on
$\partial\Omega$, which fails for the irregular velocity fields we would like to consider. We therefore consider weak solutions of \eqref{eq:classical} 
by testing against functions unrestricted at the boundary: a density--flux pair $(p, J)$  is a weak solution of the continuity equation \eqref{eq:classical} in $\Omega$ with \emph{no-flux} boundary conditions if 
\begin{equation}\label{eq:NFintro}\forall \varphi \in C^\infty(\bar\Omega\times[0,T)),\qquad 
\int_0^T\!\!\int_\Omega \big(p\,\partial_t\varphi +
J\cdot\nabla\varphi\big)\dd x\dd t
+ \int_\Omega p_0\,\varphi(\cdot,0)\dd x = 0.
\end{equation}
This formulation contains simultaneously the interior equation
$\partial_t p + \dive J = 0$ and the vanishing of the normal flux
$J\cdot\nu$ on $\partial\Omega$, with no boundary term discarded.

%\subsection{Probability  transport on a

Two distinct problems are attached to such a pair, in the
terminology of continuum mechanics: the \emph{Eulerian} description
records the evolution of fields at fixed points of the domain, while the
\emph{Lagrangian} description follows individual trajectories, and
for smooth velocities each determines the other. 

The
\emph{Eulerian problem} asks whether the density evolution $(p_t)$
is uniquely determined in a natural class; the \emph{Lagrangian
problem} asks whether the velocity field
\[
v := J/p
\]
generates a deterministic flow \emph{confined to} $\bar\Omega$ 
transporting $p_0\Ld$ onto $p_t\Ld$ --a regular Lagrangian flow
(RLF) in the sense of DiPerna--Lions \cite{DiPernaLions}. 

In the whole space, the
relation between these two descriptions was analysed in \cite{pfode}
for the probability flows arising in score-based diffusion models:
the density flow is unique under weak assumptions on the drift
\cite[Theorem~3.1]{pfode}, the probability-flow ODE is well posed
under Sobolev or BV regularity of the score together with one-sided
divergence bounds \cite[Theorem~4.5]{pfode}, and a counterexample was used to
illustrate this gap. On a bounded domain two new
phenomena enter:\begin{itemize}
    \item the behaviour of $v$ on the boundary $\partial\Omega$, and 
    \item the
question of \emph{invariance} --- the ODE carries no boundary
condition, so nothing in its formulation prevents trajectories from
leaving $\bar\Omega$. Not surprisingly, the zero-extension of the velocity field $v$ will play a role regarding this point.
\end{itemize} 

\subsection{Main results}

Our two main results identify, respectively, a sufficient boundary
criterion for the transport theorem (in fact the exact criterion
eliminating the singular boundary term in the zero-extension
approach) and the failure mode when
the boundary hypotheses are relaxed. Both revolve around a single identity.
For $v \in L^1_t BV$ near $\partial\Omega$, the extension of $v$ by
zero, $v^0 := v\,\mathbf{1}_\Omega$, is a BV field on $\R^d$ whose
distributional divergence is
\begin{equation}\label{eq:sheet}
\boxed{\ \ D\!\cdot v^0 \;=\; (\dive v)\,\Ld\llcorner\Omega
\;-\;\big(\tr v\cdot\nu\big)\,\Hd\llcorner\partial\Omega . \ \ }
\end{equation}
There is a discontinuity  at the boundary and the gradient has a singular bounded-variation (BV) part. 
Ambrosio's theory \cite{Ambrosio04,Ambrosio08} tolerates bounded-variation singularities, but the
\emph{divergence} acquires a singular sheet carried by
$\partial\Omega$ unless the normal trace of $v$ vanishes.
\emph{Tangency of $v$ is exactly the condition removing the singular
boundary part from} \eqref{eq:sheet}, and hence exactly the condition
making $v^0$ admissible for the well-posedness conditions of Ambrosio
\cite{Ambrosio04}. Our first main result converts this observation
into a transport theorem.
\begin{thmx}[Confined regular Lagrangian flow; Section
\ref{sec:mainA}]\label{thm:A}
Let $(p,J)$ be a no-flux weak solution on a bounded Lipschitz domain
satisfying the flux--density hypotheses \textup{(F0)--(F4)} of
Section \ref{sec:mainA}: positivity of $p$ in $\Omega$; $v = J/p \in
L^1\big([0,T]; BV_{\mathrm{loc}}(\Omega)\big)$ together with $v \in
L^1\big([0,T]; BV(\Omega\cap U)\big)$ for a neighbourhood $U$ of
$\partial\Omega$; absolutely continuous divergence with $[\dive v]^-
\in L^1_tL^\infty(\Omega)$; and tangency, $\tr v\cdot\nu = 0$
$\Hd$-a.e.\ on $\partial\Omega$. Then $v$ admits a regular Lagrangian
flow $Z$ on $\bar\Omega$, unique up to $\Ld$-null sets,  with
compressibility constant
$\exp\int_0^T\|[\dive v]^-\|_{L^\infty(\Omega)}$, and transporting $p_0$ to $p_t$:
\[
Z(t,\cdot)_\#\big(p_0\,\Ld\llcorner\Omega\big)
= p_t\,\Ld\llcorner\Omega \qquad \text{for every } t\in[0,T].
\]
\end{thmx}
Given \eqref{eq:sheet}, the proof is short: invariance of
$\bar\Omega$ is elementary (trajectories of $v^0$ that exited would
be frozen outside, contradicting continuity at the exit time), and
the transport identity is read off from \eqref{eq:NFintro}, which is
precisely the statement that the zero extension of $p$ solves the
global continuity equation. Tangency is automatic whenever the
density is continuous and positive up to the boundary (Corollary
\ref{cor:auto}), which covers the marginal flows of reflected
diffusions after early stopping.

Our second   result shows that the boundary assumptions are not merely an artefact of the  extension method: they cannot be removed so as to admit the boundary-current mechanism.
The construction
exhibits a vector field carrying a wall-localised mechanism of the same
nature as the singular sheet in \eqref{eq:sheet}; the precise
relation is an exhaustion statement, given in Section
\ref{sec:mainB}. The mechanism is an infinite compression at the
wall, and it destroys the regular Lagrangian flow even when
everything else --- Eulerian uniqueness, invariant characteristics,
pointwise transport --- survives.

\begin{thmx}[The boundary current; Section \ref{sec:mainB}]
\label{thm:B}
There is an explicit pair $(p,J) \in
C^\infty(\bar\Omega\times[0,T])$ on the periodic strip $\Omega =
\T\times(0,1)$ --- given in \eqref{eq:bce} below --- with $p > 0$ in
$\Omega$, vanishing linearly at the walls, solving the no-flux
continuity equation, such that:
\begin{itemize}
\item[(i)] $(p_t)$ is the unique no-flux weak solution, with initial condition
$p_0$, of an associated Fokker--Planck equation with uniformly
elliptic diffusion, in the relative energy class $\mathcal{C}_p$ of
Section  \ref{app:weighted};
\item[(ii)] the normal component $v\cdot\nu$ extends continuously to
$\bar\Omega$ with boundary values that are nonzero almost
everywhere, and outward on a moving half of each wall: tangency
fails in the classical sense. The tangential component blows up like
the inverse distance to the wall, so that $v\notin
L^1(\text{collar})$ and hypotheses \textup{(F2)--(F3)} fail there as
well;
\item[(iii)] through a.e.\ point of $\Omega$ there is a unique
integral curve of $v$; it is global, confined to $\Omega$, and the
flow so defined transports $p_0\LL{2}$ onto $p_t\LL{2}$ for every
$t$;
\item[(iv)] and yet, for every $T > 0$, $v$ admits \emph{no} regular
Lagrangian flow on $[0,T]$: for any map satisfying the trajectory
condition and any constant $L$, there are a time $t_*\in(0,T]$ and a
set of positive Lebesgue measure on which the push-forward
$Z(t_*,\cdot)_\#\LL{2}$ has density exceeding $L$ --- the density
blowing up like the inverse distance to the wall --- so that no
single compressibility constant can hold at all times.
\end{itemize}
\end{thmx}

The mechanism is a \emph{boundary current}: a tangential mass current
of nonvanishing linear density along a wall where the volume density
vanishes. The normal flux $J\cdot\nu$ vanishes because the density
does --- not because the velocity is tangential --- and near-wall
trajectories, swept along the current, are compressed against
mid-domain ones without bound. Theorem \ref{thm:B} separates three
layers that the whole-space theory tends to fuse: existence of
invariant characteristics, pointwise deterministic transport of the
marginals, and the regular Lagrangian flow property.

\subsection{Sharpness and rigidity}

To investigate the sharpness of these results, we study some  (counter)examples. Theorem \ref{thm:B} shows that the
collection of boundary hypotheses \textup{(F2)--(F4)} cannot be
relaxed so as to admit the boundary-current mechanism, the three
being structurally entangled in the smooth linear-vacuum regime, so
that no example there can isolate \textup{(F4)}; see Remark
\ref{rem:whatitshows}. We also show in Section
\ref{sec:reflected}, using a reflected Brownian motion on $[-2,2]$ started from the uniform law on $[-1,1]$, that the
no-flux evolution may be unique in the energy class, yet no regular Lagrangian flow may exist
from time zero. Here the interior, time-integrated part of hypothesis
(F2)  fails
while tangency holds at all positive times and the boundary points
carry proper integral curves: the reflecting boundary is 
'innocent', and on a bounded domain the exclusion argument acquires an
instructive twist: the extreme quantile levels \emph{are} attained
at the walls, and trajectories from the vacuum are excluded because
they would have to jump there instantaneously.

Three rigidity results (Section \ref{sec:rigidity})  delimit the
possible failure modes and show that the boundary current of Theorem
\ref{thm:B} is the only boundary mechanism available under the
stated boundedness and nondegeneracy assumptions --- they do not
classify every low-regularity boundary pathology in higher
dimension. In
dimension one, invariance can never fail, mass conservation and
no-flux pinning the monotone rearrangement inside the interval; in
any dimension, integrable total flux guarantees a superposition of
integral curves confined to $\bar\Omega$, so no counterexample can
consist in trajectories being unavailable; and a Taylor expansion
argument shows tangency is automatic wherever the velocity is bounded
near the wall, forcing the $1/\mathrm{dist}$ tangential blow-up that
the construction of Theorem \ref{thm:B} realises.

\subsection{A mimicking theorem for reflected diffusions}

As noted in Section \ref{sec:generative}, the motivation for our study comes from generative models based on reflected diffusions \cite{LouErmon,fishman2024}, where the probability flow is generated by a diffusion 
with normal reflection \cite{LouErmon} and sampling proceeds by
reversing its density flow. For the reflected SDE
\[
dX_t = f(X_t,t)\,dt + \sigma(t)\,dW_t - \nu(X_t)\,dL_t,
\qquad X_t \in \bar\Omega ,
\]
the marginal densities solve the Fokker--Planck equation with
co-normal no-flux conditions, and the probability-flow velocity is
$v = f - \tfrac12\sigma^2\nabla\log p_t$, with flux $J = p_tv$. A
structural observation (Section \ref{sec:reflected}) is that the
boundary local time $L$, which is a stochastic process, is not present in $v$: the no-flux condition reads
$p_t\,v\cdot\nu = 0$, so wherever the density is positive and
differentiable up to the boundary the field is tangential and the probability-flow ODE
requires no reflection term, in contrast with the reverse SDE, which
retains one \cite{Cattiaux88,petit1997}. Theorem \ref{thm:A} then applies:
after early stopping, for \emph{arbitrary} --- possibly singular ---
initial laws and for drifts of parabolic H\"older regularity
$f\in C^{1+\alpha,(1+\alpha)/2}$, boundary Schauder theory for the co-normal problem
makes the density $C^{2,\alpha}$ and
positive up to $\bar\Omega$, tangency holds pointwise, and the
early-stopped deterministic reflected sampler is justified exactly
as in the whole-space theory of \cite{pfode}, whose Lagrangian
theorem likewise imposes joint drift--score hypotheses. This result, which extends the result of \cite{pfode} to the case of reflected processes,  may be viewed as a ``mimicking theorem" for the density flow reflected diffusions, in the spirit of Gy\"ongy \cite{gyongy1986}.

The assumption $f\in C^{1+\alpha,(1+\alpha)/2}$ is enough to make the
zeroth-order coefficient $\dive f$ of the nondivergence form
H\"older-continuous, which is what the estimate requires, though H\"older
bounds could equally be imposed on $f$ and $\dive f$ directly.
For merely
bounded measurable drifts the density remains H\"older continuous
and positive and confined superposition solutions exist, but no
mimicking property is claimed. The sharpness results then
say precisely what can go wrong at the endpoints: the data mechanism
at $t = 0$ and, for velocity fields not arising from
uniformly elliptic diffusions with bounded drift, the boundary
current at positive times.

\subsection{A uniqueness theorem for Fokker-Planck equations with singular drift at the boundary}

The Eulerian side of Theorem \ref{thm:B} requires uniqueness for a
Fokker--Planck equation whose drift has an entrance-type singularity
$f\cdot\nu_{\mathrm{in}} \asymp \tfrac12\operatorname{dist}(x,
\partial\Omega)^{-1}$ at the boundary, with $[\dive f]^- \asymp
\operatorname{dist}^{-2}$ unbounded: the energy method behind
\cite[Theorem~3.1]{pfode} is inapplicable. Section
\ref{app:weighted} establishes uniqueness in a \emph{relative energy
class} $\{q : q/p \in L^\infty,\ \int\!\!\int p|\nabla(q/p)|^2 <
\infty\}$ by a weighted energy method whose key feature is an
\emph{exact cancellation}: testing the equation for the ratio $q/p$
against itself in $L^2(p\dd x)$, the transport terms produced by the
flux cancel against the time derivative of the weight, because $p$
solves the same continuity equation; the singular drift and its
divergence never appear, and the boundary is invisible to the energy
because the weight vanishes there. We  prove this under an
assumption (W) independent of the example, as
the uniqueness result  for no-flux Fokker--Planck equations with
entrance-type singular drifts may be of independent interest, as it does not seem to be covered by the encyclopaedic treatment of Fokker-Planck equations by Le Bris and Lions \cite{LBL2019}.

\subsection{Outline}
Section \ref{sec:prelim} fixes the setting and recalls some key results from
DiPerna--Lions--Ambrosio theory. Section \ref{sec:mainA} gives the proof of 
Theorem \ref{thm:A}. Section \ref{sec:mainB} constructs the boundary
current, states a uniqueness result for Fokker-Planck equations with no-flux boundary condition (Section \ref{app:weighted}) and proves Theorem \ref{thm:B}. Section \ref{sec:reflected}
develops the application to reflected diffusions  and the data-endpoint
sharpness result. Section \ref{sec:rigidity} proves the rigidity
results and discusses the two mechanisms. Section \ref{sec:open}
discusses some implications of our results for constrained generative models and  some further questions of interest. 

%=====================================================================
\section{Definitions and preliminary results}\label{sec:prelim}
%=====================================================================

Let $T>0$, $d\ge1$, $\Ld$ be the Lebesgue measure, $\Hd$ the
$(d-1)$-dimensional Hausdorff measure, and $\Omega\subset\R^d$  a
bounded domain with Lipschitz boundary and outer unit normal $\nu$.
Some statements require $\partial\Omega\in C^{1,1}$. We
also use the flat periodic strip $\T\times(0,1)$,
$\T=\R/2\pi\mathbb{Z}$, whose two boundary circles play the role of
$\partial\Omega$; all results transcribe to the annulus
$\{1<|x|<2\}\subset\R^2$ in polar coordinates.

\begin{definition}[No-flux weak solution]\label{def:nfweak}
A pair $(p,J)$ with $p\in L^\infty([0,T];L^1\cap L^\infty(\Omega))$,
$p_t\ge0$, $\int_\Omega p_t=1$, $t\mapsto p_t\Ld$ narrowly
continuous, and $J\in L^1([0,T]\times\Omega;\R^d)$, is a
\emph{no-flux weak solution} with initial condition $p_0$ if
\eqref{eq:NFintro} holds for every
$\varphi\in C^\infty(\bar\Omega\times[0,T))$. We refer to
\eqref{eq:NFintro} as (NF).
\end{definition}

\begin{definition}[Regular Lagrangian flow \cite{DiPernaLions}]
\label{def:rlf}
$Z:[0,T]\times\R^d\to\R^d$ is a \emph{regular Lagrangian flow} (RLF)
for a velocity field $w$ if (i) for $\Ld$-a.e.\ $x$, $t\mapsto
Z(t,x)$ is absolutely continuous with $Z(t,x)=x+\int_0^t
w(Z(s,x),s)\dd s$; and (ii) $Z(t,\cdot)_\#\Ld\le L\Ld$ for some
constant $L\ge1$ and all $t$.
\end{definition}

\begin{theorem}[Ambrosio; DiPerna--Lions
{\cite{Ambrosio04,Ambrosio08,DiPernaLions}}]\label{thm:ambrosio}
Let $w:\R^d\times(0,T)\to\R^d$ satisfy
\begin{itemize}
\item[(R1)] $w\in L^1\big((0,T);BV_{\mathrm{loc}}(\R^d;\R^d)\big)$;
\item[(R2)] $D\!\cdot w_t=(\dive w_t)\Ld$ for a.e.\ $t$, and
$[\dive w]^-\in L^1\big((0,T);L^\infty(\R^d)\big)$;
\item[(R3)] $|w|/(1+|x|)\in L^1\big((0,T);L^1\big)+
L^1\big((0,T);L^\infty\big)$.
\end{itemize}
Then: (1) for every $u_0\in L^1\cap L^\infty$ the continuity equation
$\partial_tu+\dive(wu)=0$, $u(\cdot,0)=u_0$, has at most one bounded
distributional solution in $L^\infty((0,T);L^1\cap L^\infty)$; (2)
there is an RLF $Z$ for $w$, unique up to $\Ld$-null sets, with
$Z(t,\cdot)_\#\Ld\le\exp\big(\int_0^t\|[\dive
w(\cdot,s)]^-\|_{L^\infty}\dd s\big)\Ld$; (3) for $u_0\ge0$ the
unique bounded solution is $u_t\Ld=Z(t,\cdot)_\#(u_0\Ld)$.
\end{theorem}

\begin{theorem}[Superposition {\cite{AGS,Ambrosio08}}]
\label{thm:superposition}
If $t\mapsto\mu_t$ is narrowly continuous, solves
$\partial_t\mu_t+\dive(w\mu_t)=0$ in $\R^d$, and
$\int_0^T\!\int\frac{|w|}{1+|x|}\dd\mu_t\dd t<\infty$, then there is
$\eta\in\PP\big(C([0,T];\R^d)\big)$ concentrated on absolutely
continuous integral curves of $w$ with $(e_t)_\#\eta=\mu_t$ for all
$t$, where $e_t(\gamma):=\gamma(t)$.
\end{theorem}

We shall use repeatedly the one-dimensional quantile identity (see
also \cite{pfode}, where it produces closed-form solutions of the
probability-flow ODE).

\begin{lemma}[Quantile flow]\label{lem:quantile}
Let $d=1$, $\Omega=(a,b)$, and let $(p,J)$ solve the continuity
equation with $p_t>0$ a.e., $p_tv\in L^1_{\mathrm{loc}}$, and
vanishing flux at $a$. Set $F_t(x):=\int_a^xp_t$. Then
$\partial_tF_t(x)=-p_t(x)v(x,t)$, and $t\mapsto F_t(Z_t)$ is constant
along every absolutely continuous solution of $\dot Z_t=v(Z_t,t)$;
whenever the right-hand side is defined,
$Z(t,x)=F_t^{-1}(F_0(x))$.
\end{lemma}

\begin{proof}
Integrating $\partial_tp+\partial_x(pv)=0$ over $(a,x)$ and using the
vanishing flux at $a$ gives the first identity; then along a
solution, $\frac{d}{dt}F_t(Z_t)=\partial_tF_t(Z_t)+p_t(Z_t)\dot Z_t
=0$, and inverting the strictly increasing $F_t$ gives the formula.
\end{proof}
Finally we state the Eulerian uniqueness result in the natural
energy class, the bounded-domain counterpart of
\cite[Theorem~3.1]{pfode}.

\begin{proposition}[Uniqueness in the energy class, by duality]
\label{prop:energyuniq}
Let $\Omega$ be a bounded Lipschitz domain, let $\sigma$ be
measurable with $0<\epsilon\le\sigma(t)\le\epsilon^{-1}$, and let
\[
f\in L^\infty\big(\Omega\times(0,T);\R^d\big).
\]
No hypothesis is made on $\dive f$, which may be a measure with a
nonzero singular part, and none on the boundary behaviour of $f$.
Then in
\[
\mathcal{X}_\Omega:=\Big\{p\in L^\infty\big([0,T];L^1\cap
L^\infty(\Omega)\big):\nabla p\in
L^2\big([0,T];L^2(\Omega)\big)\Big\}
\]
there is at most one no-flux weak solution of
\begin{equation}\label{eq:FPnoflux}
\partial_tp+\dive(fp)-\tfrac12\sigma(t)^2\Delta p=0,
\qquad
\big(\tfrac12\sigma^2\nabla p-fp\big)\cdot\nu=0
\ \text{ on }\partial\Omega,
\end{equation}
with a given initial condition $p_0\in L^1\cap L^\infty(\Omega)$.
\end{proposition}

The proof is by duality against the backward Kolmogorov equation,
which for a reflected diffusion carries the \emph{Neumann} condition
$\partial_\nu\varphi=0$, the adjoint of the co-normal no-flux
condition on the forward equation. The mechanism is an exact
cancellation rather than an estimate: the drift appears once from
the primal equation and once from the dual, with opposite signs, and
is never integrated by parts! This is precisely why no condition on
$\dive f$, and no normal trace of $f$, is required.

\begin{proof}
Write $a(t):=\tfrac12\sigma(t)^2\in[\tfrac12\epsilon^2,
\tfrac12\epsilon^{-2}]$, let $p^1,p^2\in\mathcal{X}_\Omega$ be
no-flux weak solutions with the same initial condition, and set $w:=p^1-p^2$,
$J_w:=fw-a\nabla w$. Since $\Omega$ is bounded, $w\in
L^\infty([0,T];L^2(\Omega))$, and $J_w\in
L^2(\Omega\times(0,T);\R^d)$ because $f\in L^\infty$ and $\nabla
w\in L^2$.

\emph{Step 1: the primal equation in variational form.} Taking
$\varphi(x,t)=\chi(x)\eta(t)$ with $\chi\in C^\infty(\bar\Omega)$
and $\eta\in C_c^\infty(0,T)$ in \eqref{eq:NFintro}, differenced for
the two solutions, gives $\partial_t w=-\dive J_w$ in the sense
that, for a.e.\ $t$,
\begin{equation}\label{eq:primalvar}
\big\langle\partial_tw(t),\chi\big\rangle
=\int_\Omega J_w(t)\cdot\nabla\chi\dd x
\qquad\text{for all }\chi\in C^\infty(\bar\Omega),
\end{equation}
and hence, by density of $C^\infty(\bar\Omega)$ in $H^1(\Omega)$ on
a Lipschitz domain and $|\langle\partial_tw,\chi\rangle|\le
\|J_w\|_{L^2}\|\nabla\chi\|_{L^2}$, for all $\chi\in H^1(\Omega)$,
with $\partial_tw\in L^2\big(0,T;H^1(\Omega)'\big)$. Note that
\eqref{eq:primalvar} holds for test functions \emph{unrestricted at
the boundary}: this is the no-flux condition, and it is the only
place where it is used. In particular $w\in C([0,T];L^2(\Omega))$
with $w(0)=0$.

\emph{Step 2: the dual problem.} Fix $\tau\in(0,T]$ and
$\psi\in L^2(\Omega)$, and consider the backward Kolmogorov problem
with Neumann boundary condition, in variational form: find
$\varphi\in L^2\big(0,\tau;H^1(\Omega)\big)$ with
$\partial_t\varphi\in L^2\big(0,\tau;H^1(\Omega)'\big)$ and
$\varphi(\tau)=\psi$ such that, for a.e.\ $t\in(0,\tau)$,
\begin{equation}\label{eq:dualvar}
\big\langle\partial_t\varphi(t),\chi\big\rangle
-a(t)\int_\Omega\nabla\varphi\cdot\nabla\chi
+\int_\Omega\big(f\cdot\nabla\varphi\big)\chi=0
\qquad\text{for all }\chi\in H^1(\Omega).
\end{equation}
Reversing time, $\tilde\varphi(t):=\varphi(\tau-t)$ solves a forward
problem governed by the family of bilinear forms
\[
\mathfrak{a}(t;u,\chi):=\tilde a(t)\int_\Omega\nabla u\cdot\nabla\chi
-\int_\Omega\big(\tilde f\cdot\nabla u\big)\chi
\qquad\text{on }H^1(\Omega)\times H^1(\Omega),
\]
which are bounded uniformly in $t$, measurable in $t$, and satisfy a
G\aa rding inequality: by Young's inequality,
\[
\mathfrak{a}(t;u,u)\ \ge\ \tilde a\|\nabla u\|_{L^2}^2
-\|f\|_{L^\infty}\|\nabla u\|_{L^2}\|u\|_{L^2}
\ \ge\ \tfrac{\epsilon^2}{4}\|\nabla u\|_{L^2}^2
-\tfrac{\|f\|_{L^\infty}^2}{\epsilon^2}\|u\|_{L^2}^2 .
\]
Lions' theorem for nonautonomous variational evolution problems
\cite{Lions61} therefore provides a unique 
$\tilde\varphi$, with $\tilde\varphi\in C([0,\tau];L^2(\Omega))$ and
$\tilde\varphi(0)=\psi$; undoing the time reversal gives $\varphi$.
Only boundedness and measurability of $f$ and $a$ are used; the
Neumann condition is natural in \eqref{eq:dualvar} and requires no
regularity of $\partial\Omega$ beyond Lipschitz, and no $H^2$
theory is invoked.

\emph{Step 3: the exact cancellation.} Both $w$ and $\varphi$ lie in
$L^2(0,\tau;H^1)$ with time derivatives in $L^2(0,\tau;(H^1)')$, so
the Lions--Magenes integration-by-parts lemma \cite{Temam} applies:
$t\mapsto\int_\Omega w\varphi$ is absolutely continuous on
$[0,\tau]$ and
\[
\int_\Omega w(\tau)\varphi(\tau)-\int_\Omega w(0)\varphi(0)
=\int_0^\tau\Big(\big\langle\partial_tw,\varphi\big\rangle
+\big\langle\partial_t\varphi,w\big\rangle\Big)\dd t .
\]
Insert $\chi=\varphi(t)\in H^1(\Omega)$ in \eqref{eq:primalvar} and
$\chi=w(t)\in H^1(\Omega)$ in \eqref{eq:dualvar}:
\[
\big\langle\partial_tw,\varphi\big\rangle
=\int_\Omega w\,f\cdot\nabla\varphi
-a\int_\Omega\nabla w\cdot\nabla\varphi,
\qquad
\big\langle\partial_t\varphi,w\big\rangle
=a\int_\Omega\nabla\varphi\cdot\nabla w
-\int_\Omega\big(f\cdot\nabla\varphi\big)w .
\]
The two drift terms are the same integral with opposite signs, and
the two diffusion terms likewise; the sum vanishes identically for
a.e.\ $t$. Observe that $f$ has been paired with $\nabla\varphi$
throughout and never differentiated: neither $\dive f$ nor a
boundary trace of $f$ appears at any point.

\emph{Step 4: conclusion.} Hence $\int_\Omega
w(\tau)\psi=\int_\Omega w(0)\varphi(0)=0$, since $w(0)=0$. As
$\psi\in L^2(\Omega)$ and $\tau\in(0,T]$ were arbitrary, $w\equiv0$.
\end{proof}

\begin{remark}[Duality versus the energy method]\label{rem:duality}
The natural alternative is the energy method used in the whole space
in \cite[Theorem~3.1]{pfode}: test the equation for $w$ against $w$
itself, obtaining
$\frac{d}{dt}\tfrac12\|w\|_{L^2}^2+a\|\nabla
w\|_{L^2}^2=\int_\Omega f\,w\cdot\nabla w$. On a bounded domain this
route costs two hypotheses that duality avoids. Since
$w\nabla w=\tfrac12\nabla(w^2)$, the drift term must be integrated
by parts, which requires $D\!\cdot f$ to be absolutely continuous
with $[\dive f]^-\in L^1_tL^\infty$ and produces, in addition, the
boundary term $\tfrac12\int_{\partial\Omega}(f\cdot\nu)w^2$ --- with
no counterpart in the whole space. That term is \emph{not}
annihilated by the no-flux condition, which constrains the total
flux $a\nabla p-fp$ and not the drift flux $fp$ separately; it must
either be given a sign ($f\cdot\nu\le0$) or absorbed through a
multiplicative trace inequality, at the price of assuming a normal
trace with $\|f\cdot\nu\|_{L^\infty(\partial\Omega)}\in L^2(0,T)$.
Neither method dominates the other: the energy method tolerates
unbounded $f$, while duality requires $f\in L^\infty$ in order to
solve the dual problem. Both, however, need \emph{global} square
integrability, $f\in L^2\big((0,T)\times\Omega\big)$, and not merely
$f\in L^2_{\mathrm{loc}}(\Omega)$: each argument tests the equation
with a function in $H^1(\Omega)$ and so requires $J_w=fw-a\nabla
w\in L^2\big((0,T)\times\Omega\big)$, hence $fw\in L^2$ up to
$\partial\Omega$, and local integrability inside $\Omega$ does not
control the behaviour of $f$ as the boundary is approached. Under
$f\in L^\infty$ this is automatic; it is the minimal hypothesis
under which the class of admissible solutions is closed under the
duality pairing. On a bounded domain, however, the drifts of interest are
bounded, and duality is the appropriate default.
\end{remark}

\begin{remark}[A singular divergence does not obstruct uniqueness]
\label{rem:signfield}
Proposition \ref{prop:energyuniq} covers drifts whose distributional
divergence is a measure. In $d=1$ with $\Omega=(-2,2)$, $\sigma\equiv
1$ and $f=-\operatorname{sign}(x)$, one has $\dive f=-2\delta_0$,
purely singular; this is the field which in the whole space defeats
the energy estimate of \cite[Remark~3.3]{pfode}, and which likewise
defeats the energy method above. Uniqueness nevertheless holds, by
duality. The distinction is worth recording: absolute continuity of
$D\!\cdot f$ is indispensable for the \emph{energy} identity ---
and, in the whole-space theory, for the dissipation bound and for
the Lagrangian hypothesis \textup{(F3)} --- but not for uniqueness
of the density flow itself. The same phenomenon appears in
\cite{pfode}, where drifts with singular divergence are admitted in
the analysis of the density lower bound through heat-kernel rather
than energy arguments.
\end{remark}

%=====================================================================
\section{Confined regular Lagrangian flow}
\label{sec:mainA}
%=====================================================================

\subsection{Boundary sheet identity}

The following  lemma plays a central role.
\begin{lemma}[Boundary sheet identity]\label{lem:sheet}
Let $\Omega$ be a bounded Lipschitz domain, $U\supset\partial\Omega$
open, and $v\in BV(\Omega\cap U;\R^d)\cap
BV_{\mathrm{loc}}(\Omega;\R^d)$. Then the zero extension
$v^0:=v\,\mathbf{1}_\Omega$ belongs to
$BV_{\mathrm{loc}}(\R^d;\R^d)$, its interior trace $\tr v$ on
$\partial\Omega$ exists in $L^1(\partial\Omega;\Hd)$, and
\begin{equation}\label{eq:sheet2}
Dv^0=Dv\llcorner\Omega-(\tr v)\otimes\nu\,\Hd\llcorner\partial\Omega,
\qquad
D\!\cdot v^0=(\dive_{\mathrm{ac}}v)\,\Ld\llcorner\Omega
+D^s\!\cdot v\llcorner\Omega
-(\tr v\cdot\nu)\,\Hd\llcorner\partial\Omega .
\end{equation}
In particular, if $D\!\cdot v=(\dive v)\Ld\llcorner\Omega$ is
absolutely continuous in $\Omega$, then $D\!\cdot v^0$ is absolutely
continuous on $\R^d$ \emph{if and only if} $\tr v\cdot\nu=0$
$\Hd$-a.e.\ on $\partial\Omega$.
\end{lemma}

\begin{proof}
The extension by zero of a $BV$ field across a Lipschitz hypersurface
is BV with the stated jump part; this is the standard trace and
extension theory for BV functions on Lipschitz domains, applied on a
finite cover of $\partial\Omega$ by balls $B\subset U$, with the
interior estimate on compacts of $\Omega$ supplied by
$BV_{\mathrm{loc}}$. Taking traces of the first identity in
\eqref{eq:sheet2} gives the second; the equivalence follows since the
two measures on the right-hand side are mutually singular.
\end{proof}
Thus tangency is exactly the condition removing the singular
boundary part $$-(\tr v\cdot\nu)\,\Hd\llcorner\partial\Omega$$ from the
divergence of the zero extension: a nonzero normal trace acts as a
singular compression or expansion sheet supported on
$\partial\Omega$, invisible to the absolutely continuous divergence
-the boundary counterpart of the absolute-continuity requirement on the divergence in the Ambrosio–DiPerna–Lions flow theorem (Theorem \ref{thm:ambrosio}). Theorem \ref{thm:B} will
exhibit a vector field for which an analogous wall-localised
mechanism produces an unbounded Jacobian along near-wall
trajectories.

\subsection{Assumptions and proof of Theorem \ref{thm:A}}

\begin{assumption}[Flux--density hypotheses (F)]\label{ass:F}
$\Omega$ is a bounded Lipschitz domain, and:
\begin{itemize}
\item[(F0)] $(p,J)$ is a no-flux weak solution (Definition
\ref{def:nfweak}) with $p_0\in L^1\cap L^\infty(\Omega)$.
\item[(F1)] $p_t>0$ $\Ld$-a.e.\ in $\Omega$ for a.e.\ $t$, so that
$v:=J/p$ is defined a.e.
\item[(F2)] $v\in L^1\big([0,T];BV_{\mathrm{loc}}(\Omega;\R^d)\big)$,
and there is an open $U\supset\partial\Omega$ with $v\in
L^1\big([0,T];BV(\Omega\cap U;\R^d)\big)$.
\item[(F3)] $D\!\cdot v_t=(\dive v_t)\Ld\llcorner\Omega$ for
a.e.\ $t$, and $[\dive v]^-\in
L^1\big([0,T];L^\infty(\Omega)\big)$.
\item[(F4)] \emph{(Tangency)} $\tr v_t\cdot\nu=0$ $\Hd$-a.e.\ on
$\partial\Omega$, for a.e.\ $t$, the trace being supplied by Lemma
\ref{lem:sheet}.
\end{itemize}
\end{assumption}

Here $L^1([0,T];BV_{\mathrm{loc}})$ means
$t\mapsto\|v_t\|_{BV(K)}\in L^1(0,T)$ for every compact
$K\subset\Omega$; global $v\in L^1_tBV(\Omega)$ implies (F2), the
interior-local form being the sharp one (Section
\ref{sec:reflected}). No separate integrability hypothesis appears:
$K:=\Omega\setminus U$ is compact in $\Omega$, so
\begin{equation}\label{eq:subsumed}
\int_0^T\!\!\int_\Omega|v|\le
\int_0^T\|v_t\|_{L^1(K)}+\int_0^T\|v_t\|_{L^1(\Omega\cap U)}<\infty:
\end{equation}
global integrability is subsumed.

\begin{proof}[Proof of Theorem \ref{thm:A}]
\emph{Step 1: the zero extension satisfies (R1)--(R3).} By Lemma
\ref{lem:sheet} applied for a.e.\ $t$, together with a fixed finite
cover of $\partial\Omega$ giving a uniform trace constant,
$v^0:=v\mathbf{1}_\Omega\in
L^1\big([0,T];BV_{\mathrm{loc}}(\R^d)\big)$: (R1) holds. By (F3),
(F4) and the equivalence in Lemma \ref{lem:sheet},
\[
D\!\cdot v^0_t=(\dive v_t)\mathbf{1}_\Omega\,\Ld,\qquad
[\dive v^0]^-=[\dive v]^-\mathbf{1}_\Omega\in
L^1\big([0,T];L^\infty(\R^d)\big):
\]
(R2) holds. Finally $v^0$ is supported in the bounded set
$\bar\Omega$ and lies in $L^1([0,T]\times\R^d)$ by
\eqref{eq:subsumed}, so (R3) holds with the $L^1$ component alone.
Theorem \ref{thm:ambrosio} yields an RLF $\hat Z$ on $\R^d$, unique
up to null sets, with the stated compression constant.

\emph{Step 2: invariance of $\bar\Omega$ (elementary).} Fix $x$ with
$\hat Z(0,x)=x\in\Omega$ and $t\mapsto\hat Z(t,x)$ an integral curve
of $v^0$. The set $O:=\{t:\hat Z(t,x)\notin\bar\Omega\}$ is open; on
any component $(a,b)\subset O$ the curve lies in the open set
$\R^d\setminus\bar\Omega$, where $v^0\equiv0$, hence is constant
$=\xi\notin\bar\Omega$ there, and by continuity $\hat
Z(a,x)=\xi\notin\bar\Omega$. If $a>0$ this places $a\in O$,
contradicting that $a$ is a left endpoint of a component; so $a=0$,
contradicting $\hat Z(0,x)\in\Omega$. Hence $O=\emptyset$. Moreover
$\int_0^T\hat
Z(t,\cdot)_\#(\Ld\llcorner\Omega)(\partial\Omega)\dd t\le
Te^{\Theta_-}\Ld(\partial\Omega)=0$, so a.e.\ trajectory spends
$\LL{1}$-null time on $\partial\Omega$; since $v^0=v$
$\Ld$-a.e.\ on $\Omega$, the integral identity holds with $v$ in
place of $v^0$. Set $Z:=\hat Z|_{[0,T]\times\bar\Omega}$.
Uniqueness: any flow as in the statement, extended by the constant
flow outside $\bar\Omega$, is an RLF for $v^0$, hence coincides with
$\hat Z$ up to null sets.

\emph{Step 3: transport.} Let $\bar p_t:=p_t\mathbf{1}_\Omega$. For
$\varphi\in C_c^\infty(\R^d\times[0,T))$, using $J=pv$
a.e.\ in $\Omega$ and $v^0=v$ on $\Omega$,
\[
\int_0^T\!\!\int_{\R^d}\bar p\big(\partial_t\varphi+
v^0\cdot\nabla\varphi\big)+\int_{\R^d}\bar p_0\varphi(\cdot,0)
=\int_0^T\!\!\int_\Omega\big(p\,\partial_t\varphi+
J\cdot\nabla\varphi\big)+\int_\Omega p_0\varphi(\cdot,0)=0
\]
by (NF): the no-flux formulation with unrestricted test functions is
\emph{exactly} the statement that the zero extension of $p$ solves
the global continuity equation --- no boundary term appears because
none was ever discarded. Thus $\bar p$ is a bounded distributional
solution for $v^0$ (with $\bar pv^0=J\mathbf{1}_\Omega\in L^1$), as
is $u_t:=\hat Z(t,\cdot)_\#(\bar p_0\Ld)$ by Theorem
\ref{thm:ambrosio}(3); comparison gives $u_t=\bar p_t\Ld$ for
a.e.\ $t$, Step 2 shows $u_t$ is carried by $\bar\Omega$, and narrow
continuity of both curves upgrades the identity to every $t$.
\end{proof}

\subsection{Tangency and time
reversal}

\begin{corollary}[Tangency is automatic for positive continuous
densities]\label{cor:auto}
Assume \textup{(F0)--(F3)} and that for a.e.\ $t$, $p_t\in
C^1(\bar\Omega)$ with $\min_{\bar\Omega}p_t>0$. Then \textup{(F4)}
holds and Theorem \ref{thm:A} applies.
\end{corollary}

\begin{proof}
For $\psi\in C^1(\bar\Omega)$ write $\psi=\chi\psi+(1-\chi)\psi$ with
$\chi\in C_c^\infty(U)$, $\chi\equiv1$ near $\partial\Omega$; the
$(1-\chi)$-part is an interior identity requiring no trace. For the
$\chi$-part: for a.e.\ $t$, $J_t=p_tv_t\in BV(\Omega\cap U)$ with
$\tr J_t=p_t|_{\partial\Omega}\tr v_t$ $\Hd$-a.e. Testing (NF) with
$\varphi=\chi\psi(x)\eta(t)$, $\eta\in C_c^\infty(0,T)$, and
subtracting the interior identity $\partial_tp=-\dive J$ in
$\mathcal{D}'(\Omega\times(0,T))$ via the Gauss--Green formula for
BV fields on $\Omega\cap U$ \cite{Anzellotti,ChenFrid} --- only the
$\partial\Omega$ portion of the boundary contributing, $\chi\psi$
vanishing on the rest --- yields
$\int_0^T\eta\int_{\partial\Omega}\chi\psi\,\tr
J_t\cdot\nu\,d\Hd\dd t=0$ for all $\psi,\eta$; hence $\tr
J_t\cdot\nu=0$ $\Hd$-a.e.\ for a.e.\ $t$, and dividing by
$p_t|_{\partial\Omega}>0$ gives (F4).
\end{proof}

\begin{remark}[Sharpness of the assumptions]\label{rem:sharp}
(a) Within the zero-extension strategy the collar $BV$ in (F2)
cannot be reduced to boundedness plus a divergence-measure condition:
the Anzellotti--Chen--Frid normal trace \cite{Anzellotti,ChenFrid}
would still make (F4) meaningful, but (R1) fails near
$\partial\Omega$, and renormalisation is required to operate
\emph{across} $\partial\Omega$, where for merely bounded
divergence-free fields uniqueness fails \cite{Depauw}. Sobolev
regularity $W^{1,1}(\Omega\cap U)$ serves equally well. This is a
limitation of the method, however, and not of the problem: as
discussed in Remark \ref{rem:CDIN} below, the intrinsic theory on
$\Omega$ developed by Crippa, De Rosa, Inversi and Nesi
\cite{CDRIN} dispenses with boundary $BV$ regularity precisely on
the exiting and tangent portions of $\partial\Omega$, which is where
the present hypotheses place us. (b) The
time-integrability in the interior part of (F2) is essential: in
the sharpness example of Section \ref{sec:reflected} the velocity is
real-analytic in $\Omega$ for every $t>0$, and it is precisely
$\int_0^T\|v_t\|_{BV(K)}\dd t$ that diverges. (c) The boundary
hypotheses \textup{(F2)--(F4)} cannot be relaxed jointly so as to
admit a boundary current: Theorem \ref{thm:B}. Whether tangency
alone is indispensable is a subtler question; see Remark
\ref{rem:whatitshows}.
\end{remark}

\begin{remark}[Normal traces and
the intrinsic theory on $\Omega$]\label{rem:CDIN}
Various notions of normal trace have been studied in the literature:
the distributional one (Anzellotti \cite{Anzellotti}, Chen--Frid
\cite{ChenFrid}) for measure-divergence fields; the
\emph{normal Lebesgue trace} $v^{\partial\Omega}_n$ introduced by De
Rosa and Inversi \cite{DeRosaInversi}, which requires the boundary
value to be attained in an averaged Lebesgue sense on interior
balls; and the strong $BV$ trace of  Ambrosio et al. \cite[Thm.~3.87]{AFP}, used in
Lemma \ref{lem:sheet}. Crippa et al. \cite{CDRIN} prove that for bounded
measure-divergence fields the normal Lebesgue trace, when it exists,
coincides with the distributional one and satisfies the Gauss--Green
identity, and that it is a notion lying \emph{strictly} between the
other two. Since a $BV$ field attains its normal Lebesgue trace, and
that trace is $v^\Omega\cdot\nu$ \cite[Prop.~5.5]{DeRosaInversi},
our hypothesis \textup{(F4)} implies
$v^{\partial\Omega}_n\equiv0$: tangency as used here is the
strongest of the three, and \emph{Lebesgue tangency} is its natural
weakening.

Crippa et al. \cite{CDRIN} recently established uniqueness for
continuity equations on bounded domains without the global
boundary $BV$ assumption of \cite{CDS} exactly on the portion of
$\partial\Omega$ where characteristics exit or are tangent. In their
notation, tangency forces $\Gamma^-=\emptyset$ --- no characteristic
enters --- so their $BV$-near-$\Gamma^-$ hypothesis is vacuous,
while their uniformly-exiting condition follows from vanishing
normal Lebesgue trace \cite[Cor.~3.9]{CDRIN}. Their
\cite[Cor.~4.6]{CDRIN}, for divergence-free fields tangent to
$\partial\Omega$ in the Lebesgue sense, gives uniqueness under
interior $BV_{\mathrm{loc}}(\Omega)$ regularity alone. This is
strong evidence that the collar hypothesis in \textup{(F2)} is
removable for the Eulerian portion of Theorem \ref{thm:A}; what does
not follow automatically is the Lagrangian result, since our flow,
invariance and transport statements are all obtained from the zero
extension on $\R^d$ rather than intrinsically on $\Omega$.  Two hypotheses are not comparable,
though: \cite{CDRIN} work with $v\in L^\infty$ and $\dive v\in
L^1_tL^\infty_x$ two-sided, whereas \textup{(F2)--(F3)} allow
unbounded $BV$ fields and only a one-sided divergence bound.
Conversely, the counterexample of Theorem \ref{thm:B} is unbounded
near the wall and therefore lies outside the framework of
\cite{CDRIN} as well; it contradicts none of their uniqueness
results.

Crippa et al. \cite{CDRIN} construct a bounded divergence-free field whose
normal Lebesgue trace exists and equals $-1$ --- characteristics
entering everywhere --- for which uniqueness fails. Tangency
excludes exactly this configuration, which is why the present
setting falls on the favourable side of their dichotomy.
\end{remark}

\begin{remark}[Time reversal]\label{rem:timerev}
If in addition $[\dive v]^+\in L^1([0,T];L^\infty(\Omega))$, the
two-sided compression bound and essential invertibility of
$Z(t,\cdot)$ follow as in the whole space; tangency is invariant
under $v\mapsto-v$, so time reversal requires no new boundary
hypothesis, and the deterministic reverse flow transports $p_T$ onto
$p_0$ under the same conditions read backwards.
\end{remark}

%=====================================================================
\section{The boundary current}\label{sec:mainB}
%=====================================================================

We now prove Theorem \ref{thm:B}. The construction is guided by the
rigidity results of Section \ref{sec:rigidity}, which show (Lemma
\ref{lem:taylor} below) that tangency is automatic wherever $v$ is
bounded near the wall. Any example in which tangency fails must
therefore carry a \emph{boundary current}: a tangential flux of
nonvanishing linear density
$g(x_1,t):=J_1(x_1,0,t)$ along a wall where $p$ vanishes linearly,
with
\begin{equation}\label{eq:bcstructure}
v_1\approx\frac{g(x_1,t)}{c(x_1,t)\,x_2},\qquad
v_2\big|_{x_2=0}=-\frac{\partial_1g(x_1,t)}{c(x_1,t)}
\qquad (p\approx c(x_1,t)\,x_2) .
\end{equation}

\subsection{Construction of the boundary current}

Let $\Omega:=\T\times(0,1)$. Fix $\lambda>0$, set $\xi:=x_1-t$,
$m(s):=1+\tfrac12\cos s\in[\tfrac12,\tfrac32]$, and define
\begin{equation}\label{eq:bce}
p(x,t):=\tfrac14\sin(\pi x_2)\,m(\xi),\qquad
J:=p\,e_1+\nabla^\perp\psi,\qquad
\psi(x,t):=\tfrac{\lambda}{\pi}\sin(\pi x_2)\cos\xi,
\end{equation}
where $\nabla^\perp\psi:=(\partial_2\psi,-\partial_1\psi)=
\big(\lambda\cos(\pi x_2)\cos\xi,\ \tfrac{\lambda}{\pi}\sin(\pi
x_2)\sin\xi\big)$. Then $p\in C^\infty(\bar\Omega\times[0,T])$ is
positive in $\Omega$, vanishes linearly on both walls, and
$\int_\Omega p_t=\tfrac14\cdot2\pi\cdot\tfrac2\pi=1$. The continuity
equation holds classically ($\partial_tp=-\partial_1p$ while
$\dive J=\partial_1p+\dive\nabla^\perp\psi=\partial_1p$), and the
no-flux condition holds pointwise
($J_2=\tfrac{\lambda}{\pi}\sin(\pi x_2)\sin\xi$ vanishes at
$x_2\in\{0,1\}$). The velocity is
\begin{equation}\label{eq:bcv}
v_1=1+\frac{4\lambda\cos(\pi x_2)\cos\xi}{\sin(\pi x_2)\,m(\xi)},
\qquad
v_2=\frac{4\lambda\sin\xi}{\pi\,m(\xi)} .
\end{equation}
Both components are $C^\infty$ in $\Omega$; $v_2$ extends smoothly to
$\bar\Omega$ with trace $\tfrac{4\lambda\sin(x_1-t)}{\pi m(x_1-t)}$
on $\{x_2=0\}$, nonzero for a.e.\ $(x_1,t)$ and \emph{outward} on the
moving half-wall $\{\sin(x_1-t)<0\}$. The tangential component
realises the boundary current,
$v_1\sim\tfrac{4\lambda\cos\xi}{\pi m\,x_2}$, matching
\eqref{eq:bcstructure}; consequently $v\notin L^1$ near either wall
and $[\dive v]^-$ is unbounded there. This proves part (ii) of
Theorem \ref{thm:B}: tangency fails classically, and hypotheses
\textup{(F2)--(F3)} fail on the collar as well --- consistently with
Theorem \ref{thm:A}, which would otherwise apply and contradict part
(iv).

It should be emphasised that Lemma \ref{lem:sheet} is
\emph{not} available here, and that the failure of tangency cannot
be expressed through it. That lemma presupposes $v\in
BV(\Omega\cap U)$, hence $v\in L^1$ on a collar; since $v_1\sim
C/x_2$, the zero extension $v\mathbf{1}_\Omega$ is not locally
integrable on $\R^d$, does not define a distribution, and has no
distributional divergence --- so there is no ``boundary sheet''
$-(\tr v\cdot\nu)\Hd\llcorner\partial\Omega$ associated with this
field in the sense of \eqref{eq:sheet}. What is available is the
classical trace of $v\cdot\nu$, which exists because $v_2$ happens
to extend smoothly, and the following exhaustion statement, which is
the precise form of the analogy. For $\eta>0$ let
$\Omega_\eta:=\T\times(\eta,1-\eta)$. On $\bar\Omega_\eta$ the field
$v$ is smooth and bounded with bounded derivatives, so Lemma
\ref{lem:sheet} applies to $\Omega_\eta$ and gives
\[
D\!\cdot\big(v\mathbf{1}_{\Omega_\eta}\big)
=(\dive v)\,\LL{2}\llcorner\Omega_\eta
\;-\;\big(v\cdot\nu_\eta\big)\,
\mathcal{H}^{1}\llcorner\partial\Omega_\eta ,
\]
whose singular part on the lower face $\{x_2=\eta\}$ equals
$\tfrac{4\lambda\sin\xi}{\pi m(\xi)}\mathcal{H}^1$, independently of
$\eta$. These singular parts do not vanish as $\eta\downarrow0$;
they converge weakly-$*$ to the nonzero measure
$\tfrac{4\lambda\sin\xi}{\pi m(\xi)}\,\mathcal{H}^1\llcorner\{x_2=0\}$.
The example therefore exhibits, on every interior surface parallel
to the wall, exactly the singular divergence contribution that
tangency excludes --- while the limiting object is not itself the
divergence of any extension of $v$.
\subsection{Weighted uniqueness for no-flux Fokker--Planck equations
with entrance-type singular drifts}\label{app:weighted}
%=====================================================================

In this section we establish a uniqueness result 
for no-flux
Fokker--Planck equations whose drift has an entrance-type
singularity at the boundary, a regime where $f\notin L^\infty$ and
$[\dive f]^-\notin L^\infty$, so that neither the duality argument
of Proposition \ref{prop:energyuniq} nor the unweighted energy
method of Remark \ref{rem:duality} apply.

We will use this result in the sequel (
Corollary \ref{cor:bcuniq}), but as it is of independent interest we formulate in more general form.
\begin{assumption}[W]\label{ass:W}
$\Omega$ is a bounded domain with $C^{1,1}$ boundary, or the
periodic strip $\T\times(0,1)$ (or annulus). The pair
$p\in C^1(\bar\Omega\times[0,T])$,
$J\in C^1(\bar\Omega\times[0,T];\R^d)$ satisfies, classically,
\[
\partial_tp+\dive J=0\ \text{ in }\bar\Omega\times[0,T],\qquad
J\cdot\nu=0\ \text{ on }\partial\Omega\times[0,T],
\]
with $p>0$ in $\Omega\times[0,T]$ and
$p(x,t)\le C_0\operatorname{dist}(x,\partial\Omega)$. Set $v:=J/p$,
$\sigma\equiv1$, and $f:=v+\tfrac12\nabla\log p$ on $\Omega$, so
that $(p,J)$ solves the no-flux Fokker--Planck problem with
coefficients $(f,\sigma)$, by the identity $fp-\tfrac12\nabla p=J$.
\end{assumption}
No lower bound  $p\ge c\operatorname{dist}$ is required.

\begin{definition}[Relative energy class]\label{def:relclass}
\[
\mathcal{C}_p:=\Big\{q\ge0\ :\ h:=q/p\in
L^\infty\big((0,T)\times\Omega\big),\ \
\int_0^T\!\!\int_\Omega p\,|\nabla h|^2\dd x\dd t<\infty\Big\},
\]
and $q\in\mathcal{C}_p$ is a no-flux weak solution with initial condition $q_0$
if \eqref{eq:NFintro} holds with the flux
\[
J_q:=fq-\tfrac12\nabla q,
\]
this being the flux of the Fokker--Planck equation written as a
conservation law, $\partial_tq+\dive J_q=0$, and hence the field
entering \eqref{eq:NFintro}.
\end{definition}

Two remarks make this meaningful. First, in terms of $h$ the flux is
\begin{equation}\label{eq:fluxid}
J_q=\big(v+\tfrac12\nabla\log p\big)hp-\tfrac12\nabla(hp)
=h\,J-\tfrac12\,p\,\nabla h,
\end{equation}
consistently with $J_p=J$ for $h\equiv1$,
in which every singular object is absorbed:
$p\nabla h\in L^2((0,T)\times\Omega)$ since
$\int|p\nabla h|^2\le\|p\|_\infty\int p|\nabla h|^2$, and
$hJ\in L^\infty$; all terms of the weak formulation are finite even
though $f\notin L^2(\Omega)$. Second,
$p\in\mathcal{C}_p$ (with $h\equiv1$) and is a no-flux weak solution
with initial condition $p_0$, by (W) and the classical divergence theorem.

\begin{theorem}[Weighted uniqueness]\label{thm:weighted}
Under Assumption \textup{(W)}, for every initial condition there is at
most one no-flux weak solution in $\mathcal{C}_p$. In particular,
$(p_t)$ is the unique no-flux weak solution in $\mathcal{C}_p$ with
initial condition $p_0$.
\end{theorem}

The proof rests on the observation that testing the equation
for the ratio $h=q/p$ against $h$ in the weighted space
$L^2(p\dd x)$ produces transport terms that cancel \emph{exactly}
against the time derivative of the weight, because $p$ solves the
same continuity equation. The singular drift never appears, and
neither does its divergence.

\begin{proof}
Let $q^1,q^2\in\mathcal{C}_p$ solve with the same initial condition; set
$u:=q^1-q^2$, $w:=u/p$, $M:=\|w\|_\infty$. Subtracting the weak
formulations and using \eqref{eq:fluxid},
\begin{equation}\label{eq:diffweak}
\int_0^T\!\!\int_\Omega p\,w\,\partial_t\varphi
=\int_0^T\!\!\int_\Omega\Big(\tfrac12\,p\nabla w-
wJ\Big)\cdot\nabla\varphi
\qquad\forall\,\varphi\in C^\infty(\bar\Omega\times[0,T)).
\end{equation}
\emph{Step 1: localisation.} Write
$\rho(x):=\operatorname{dist}(x,\partial\Omega)$, which is $C^{1,1}$
in a collar $\{\rho<\rho_0\}$ with $\nabla\rho=-\nu$ on
$\partial\Omega$. For $\delta\in(0,\rho_0/4)$ pick
$\chi_\delta:=\theta_\delta\circ\rho$ with
$\theta_\delta\in C^\infty([0,\infty);[0,1])$, $\theta_\delta=0$ on
$[0,\delta]$, $\theta_\delta=1$ on $[2\delta,\infty)$,
$|\theta'_\delta|\le C\delta^{-1}$ (on the strip take
$\chi_\delta=\chi_\delta(x_2)$ directly). Set
$\Omega_\delta:=\{\rho>\delta\}$, on whose closure $p$ is bounded
above and below by positive constants uniformly in $t$. Then
$q^i\in\mathcal{C}_p$ gives
$u\in L^2((0,T);H^1(\Omega_\delta))$, and by \eqref{eq:diffweak} ---
whose right-hand side is a bounded functional of
$\nabla\varphi\in L^2$ --- also $\partial_tu\in
L^2((0,T);H^{-1}(\Omega_\delta))$. By density (both sides of
\eqref{eq:diffweak} being continuous in the relevant norms, using
$p\nabla w\in L^2$ and $wJ\in L^\infty$), \eqref{eq:diffweak}
extends to test functions in $L^2((0,T);H^1_0(\Omega_\delta))$ with
square-integrable time derivative vanishing at $T$; and the standard
theory (Lions--Magenes duality \cite{Temam}) yields
$u\in C([0,T];L^2(\Omega_\delta))$ with $u(0)=0$, the initial data
being equal.

\emph{Step 2: weighted energy identity.} Let
$\zeta:=\chi_\delta/p\in C^1(\bar\Omega_\delta\times[0,T])$,
bounded with bounded time derivative there. The chain rule for
$\langle\partial_tu,\zeta u\rangle$ (legitimate since
$\zeta u=w\chi_\delta\in L^2((0,T);H^1_0(\Omega_\delta))$) gives,
for a.e.\ $t$,
\[
\frac{d}{dt}\int_\Omega p\,w^2\chi_\delta
=2\big\langle\partial_tu,\,w\chi_\delta\big\rangle
-\int_\Omega\partial_tp\;w^2\chi_\delta .
\]
Substituting \eqref{eq:diffweak} with test function $w\chi_\delta$
and expanding
$\nabla(w\chi_\delta)=\chi_\delta\nabla w+w\nabla\chi_\delta$,
\[
2\big\langle\partial_tu,w\chi_\delta\big\rangle
=-\int_\Omega p|\nabla w|^2\chi_\delta
+2\int_\Omega wJ\cdot\nabla w\,\chi_\delta+E_\delta(t),
\quad
E_\delta:=-\!\int_\Omega p(\nabla w\cdot\nabla\chi_\delta)w
+2\!\int_\Omega w^2J\cdot\nabla\chi_\delta .
\]

\emph{Step 3: exact cancellation.} On
$\operatorname{supp}\chi_\delta$ the function $w$ lies in
$H^1\cap L^\infty$, so $w^2\in W^{1,1}$ with
$\nabla(w^2)=2w\nabla w$; $J$ is $C^1$; integrating by parts against
$\chi_\delta\in C_c^\infty$ and using $\partial_tp=-\dive J$
classically,
\[
2\int_\Omega wJ\cdot\nabla w\,\chi_\delta
-\int_\Omega\partial_tp\,w^2\chi_\delta
=\int_\Omega\dive\!\big(J\,w^2\big)\chi_\delta
=-\int_\Omega(J\cdot\nabla\chi_\delta)\,w^2 .
\]
This is the heart of the proof: the transport term produced by the
flux and the term produced by differentiating the weight combine
into a pure divergence, which sees only the component
$J\cdot\nabla\chi_\delta\propto J\cdot\nabla\rho$ on the cutoff
layer. The individually divergent quantities (note
$\int|J|^2/p=\infty$ in the example of Section \ref{sec:mainB}) are
never separated. Collecting the terms we obtain
\begin{equation}\label{eq:enid}
\frac{d}{dt}\int_\Omega p\,w^2\chi_\delta
=-\int_\Omega p|\nabla w|^2\chi_\delta+\widetilde E_\delta(t),
\qquad
\widetilde E_\delta:=E_\delta-\int_\Omega
(J\cdot\nabla\chi_\delta)\,w^2 .
\end{equation}

\emph{Step 4: the layer errors vanish.} Let
$L_\delta:=\{\delta<\rho<2\delta\}\supset
\operatorname{supp}\nabla\chi_\delta$, of measure $\le C\delta$.
Three facts are used: $p\le2C_0\delta$ on $L_\delta$ (by assumption);
$|J\cdot\nabla\rho|\le C_1\rho\le2C_1\delta$ on $L_\delta$, which
follows from $J\in C^1(\bar\Omega)$, $\rho\in C^{1,1}$, and
$J\cdot\nabla\rho=-J\cdot\nu=0$ on $\partial\Omega$ --- the no-flux
condition doing its work; and, decisively, that
$\nabla\chi_\delta$ is parallel to $\nabla\rho$, so the cutoff never
meets the \emph{tangential} component of $J$, the carrier of any
boundary current. Then, integrating in time,
\[
\Big|\int_0^T\!\!\int w^2(J\cdot\nabla\chi_\delta)\Big|
\le M^2\frac{C}{\delta}\cdot2C_1\delta\cdot T|L_\delta|
\le C'M^2T\,\delta\xrightarrow[\delta\downarrow0]{}0,
\]
and by Cauchy--Schwarz with weight $p$,
\[
\Big|\int_0^T\!\!\int p(\nabla w\cdot\nabla\chi_\delta)w\Big|
\le\frac{CM}{\delta}
\Big(\int_0^T\!\!\int_{L_\delta}p|\nabla w|^2\Big)^{1/2}
\Big(T|L_\delta|\sup_{L_\delta}p\Big)^{1/2}
\le C''M\sqrt{T}\,\varepsilon_\delta,
\]
where $\varepsilon_\delta:=(\int_0^T\!\int_{L_\delta}p|\nabla
w|^2)^{1/2}\to0$ as the tail of a convergent integral. The scaling
deserves emphasis: the factor $\delta^{-1}$ from the cutoff is
beaten by $(\int_{L_\delta}p)^{1/2}\lesssim\delta$, which is
precisely where $p\lesssim\operatorname{dist}$ is used: the
weight vanishes fast enough at the wall that the boundary is
invisible to the energy.

Discarding the negative dissipation in
\eqref{eq:enid} and integrating from $0$ to $t$ with $u(0)=0$,
\[
\int_\Omega p_t\,w_t^2\,\chi_\delta\le
C\big(M\sqrt{T}\,\varepsilon_\delta+M^2T\,\delta\big)
\quad\text{for a.e.\ }t .
\]
Letting $\delta\downarrow0$: monotone convergence on the left, zero
on the right; hence $w_t=0$ $p_t$-a.e., and since $p_t>0$ in
$\Omega$, $q^1=q^2$ a.e.
\end{proof}

\begin{remark}[Interpretation]\label{rem:chisq}
From a probabilistic perspective, $h=q/p$ is a Doob $h$-transform ratio, and the
identity of Steps 2--3 is the decay of the $\chi^2$-divergence
$\int(q/p-1)^2p\dd x$ with dissipation the relative Dirichlet energy
$\int p|\nabla(q/p)|^2$; the class $\mathcal{C}_p$ is the natural
relative analogue of the energy class $\mathcal{X}_\Omega$ of
Proposition \ref{prop:energyuniq}. A variant of the same computation
applied to $(w-k)_+^2$ yields the comparison principle
$q_0\le kp_0\Rightarrow q_t\le kp_t$. The domination $q\le Cp$
forbids competing solutions from carrying mass toward the walls
faster than $p$ does; uniqueness in the unweighted class remains
open (Section \ref{sec:open}).
\end{remark}

\subsection{Eulerian perspective: uniqueness of the density evolution}

Setting $\sigma\equiv1$ and $f:=v+\tfrac12\nabla\log p$, the pair
$(p,J)$ satisfies, identically, the no-flux Fokker--Planck problem
with coefficients $(f,\sigma)$, by the identity
$fp-\tfrac12\nabla p=J$. Near the walls, $f$ has the Bessel-type
entrance singularity $f_2\approx\tfrac{1}{2x_2}$ (the associated
diffusion never reaches the wall and no local time appears), and
$f\notin L^\infty(\Omega\times(0,T))$, while
$[\dive f]^-\asymp\operatorname{dist}(x,\partial\Omega)^{-2}$ is
unbounded: neither Proposition \ref{prop:energyuniq} nor its energy
counterpart of Remark \ref{rem:duality} is applicable. The
Assumption (W) of Section  \ref{app:weighted} holds by inspection
($p\le\tfrac{3\pi}{8}\operatorname{dist}(x,\partial\Omega)$),
whence:
\begin{corollary}\label{cor:bcuniq} The density flow
$(p_t)$ of \eqref{eq:bce} is the unique no-flux weak solution for the  Fokker--Planck equation with initial condition $p_0$,  in the 
relative energy class $\mathcal{C}_p$ of Theorem \ref{thm:weighted}.
\end{corollary}
This proves part \textup{(i)} of Theorem \ref{thm:B}.
\subsection{Characteristics, marginal transport, and failure of compressibility}

We now give the proof of  (iii) and (iv).
We isolate the quantitative mechanism as a lemma. Let 
$m(\xi)=1+\tfrac12\cos\xi\in[\tfrac12,\tfrac32]$, and 
\begin{equation}\label{eq:speeds}
c_\lambda:=\frac{4\lambda}{3\pi},
\qquad
C_\lambda:=\frac{8\lambda}{\pi},
\qquad
\bar t:=\frac{1}{8C_\lambda}=\frac{\pi}{64\lambda} .
\end{equation}
From $v_2=\tfrac{4\lambda\sin\xi}{\pi m(\xi)}$ and
$\tfrac12\le m\le\tfrac32$ one has the two-sided control
\begin{equation}\label{eq:twosided}
|v_2|\le C_\lambda \ \text{ everywhere},
\qquad
v_2\ge c_\lambda \ \text{ wherever }\sin\xi\ge\tfrac12.
\end{equation}
The upper bound  confines the trajectory and the lower
bound  makes it climb.

\begin{lemma}[Transit lemma]\label{lem:transit}
Let $t_*\in(0,\bar t\,]$ and, for $0<\kappa<\tfrac{1}{100}$, put
\[
B_\kappa:=\Big\{(\xi,x_2):\ \xi\in\big(\tfrac\pi6,\tfrac\pi3\big),
\ \ \kappa<\bar\psi(\xi,x_2)<2\kappa,\ \ x_2<\tfrac12\Big\},
\]
a nonempty open subset of $\Omega$ in the co-moving frame. Then:
\begin{itemize}
\item[(a)] $p\le\tfrac32\kappa$ on $B_\kappa$;
\item[(b)] every frame trajectory starting in $B_\kappa$ remains in
$\{x_2<\tfrac12\}$ on $[0,t_*]$, with $\sin\xi\ge\tfrac12$
throughout, and satisfies $x_2(t_*)\ge c_\lambda t_*$;
\item[(c)] consequently $p\big(Z(t_*,z)\big)\ge
c(t_*):=\tfrac18\sin(\pi c_\lambda t_*)>0$ for every $z\in
B_\kappa$, with $c(t_*)$ independent of $\kappa$;
\item[(d)] $A_\kappa:=Z(t_*,B_\kappa)$ is open with
$\LL{2}(A_\kappa)>0$.
\end{itemize}
\end{lemma}

\begin{proof}
(a) On $B_\kappa$, $\cos\xi\in(\tfrac12,\tfrac{\sqrt3}{2})$, so
$\sin(\pi x_2)=\bar\psi/\cos\xi<2\kappa/\tfrac12=4\kappa<\tfrac1{25}$;
hence $x_2<\tfrac1{50}$ and
$p=\tfrac14\sin(\pi x_2)m(\xi)\le\tfrac14\cdot4\kappa\cdot\tfrac32
=\tfrac32\kappa$.

(b) On the orbit through such a point, $\cos\xi=\bar\psi/\sin(\pi
x_2)>0$, so $\xi\in(-\tfrac\pi2,\tfrac\pi2)$ for all time. As long as
$x_2<\tfrac12$ we have $\dot\xi=\tfrac{4\lambda\cos(\pi
x_2)\cos\xi}{\sin(\pi x_2)m}>0$, so $\xi$ increases and remains in
$[\tfrac\pi6,\tfrac\pi2)$; hence $\sin\xi\ge\tfrac12$ and, by
\eqref{eq:twosided}, $c_\lambda\le\dot x_2\le C_\lambda$. Integrating
the \emph{upper} bound, $x_2(t)\le
x_2(0)+C_\lambda t\le\tfrac1{50}+C_\lambda\bar t=\tfrac1{50}
+\tfrac18<\tfrac12$ for $t\le t_*$, so the constraint $x_2<\tfrac12$
is never violated and the bootstrap closes by continuity. Integrating
the \emph{lower} bound gives $x_2(t_*)\ge c_\lambda t_*$.

(c) By (b), $x_2(t_*)\in[c_\lambda t_*,\tfrac12]$, so $\pi
x_2(t_*)\in[\pi c_\lambda t_*,\tfrac\pi2]$ and, $\sin$ being
increasing there, $\sin(\pi x_2(t_*))\ge\sin(\pi c_\lambda t_*)$;
with $m\ge\tfrac12$, $p\ge\tfrac14\sin(\pi c_\lambda
t_*)\cdot\tfrac12=c(t_*)$.

(d) By (b) the trajectories issuing from $B_\kappa$ stay in the open
region $\{0<x_2<\tfrac12,\ \cos\xi>0\}$ on $[0,t_*]$, where the frame
field $v^{\mathrm{fr}}=\nabla^\perp\psi/p$ is $C^\infty$. Hence
$Z(t_*,\cdot)$ is a $C^\infty$ diffeomorphism of a neighbourhood of
$\bar B_\kappa$ onto its image, by smooth dependence on initial
conditions and invertibility of the backward flow; the image of the
nonempty open set $B_\kappa$ is therefore open and of positive
measure.
\end{proof}

\begin{proof}[Proof of Theorem \ref{thm:B}, parts (iii)--(iv)]
\emph{(iii).} In the co-moving frame $(\xi,x_2)$ the ODE $\dot
Z=v(Z,t)$ is autonomous: $\dot\xi=v_1-1$, $\dot x_2=v_2$ define the
field $v^{\mathrm{fr}}=\nabla^\perp\psi/p$, with $\psi,p$ read as
functions of $(\xi,x_2)$. Then
$\tfrac{d}{dt}\psi(\xi(t),x_2(t))=\nabla\psi\cdot
v^{\mathrm{fr}}=0$: $\bar\psi:=\sin(\pi x_2)\cos\xi$ is a first
integral, and orbits lie on its level sets. The level set
$\{\bar\psi=0\}$ is the union of the two walls and the two vertical
circles $\{\cos\xi=0\}$, a Lebesgue-null set; every other level
$\{\bar\psi=\kappa\}$, $0<|\kappa|<1$, is a closed curve in $\Omega$
encircling $(\xi,x_2)=(0,\tfrac12)$ (for $\kappa>0$; symmetrically
for $\kappa<0$), with a lower arc
$x_2=\tfrac1\pi\arcsin(\kappa/\cos\xi)$ hugging the bottom wall, an
upper arc hugging the top wall, and turning points at
$\cos\xi=\kappa$, $x_2=\tfrac12$. On these curves $v$ is smooth and
nonvanishing: solutions are unique, periodic, and confined. The frame
field satisfies $\dive(p\,v^{\mathrm{fr}})=\dive\nabla^\perp\psi=0$,
so the frame flow preserves the measure $p\dd\xi\dd x_2$;
translating to the lab frame gives
$Z(t,\cdot)_\#(p_0\LL{2})=p_t\LL{2}$.

\emph{(iv).} Let $Z$ satisfy the trajectory condition of Definition
\ref{def:rlf}(i). Off the null set $\{\bar\psi=0\}$ the field is
smooth, so $Z$ agrees a.e.\ with the flow of (iii); it suffices to
disprove the compressibility bound for that flow. Since the frame
flow preserves $p\,\LL{2}$, the push-forward
$Z(t,\cdot)_\#\LL{2}$ has density
\[
\rho_t(y)=\frac{p(y)}{p\big(Z(t,\cdot)^{-1}(y)\big)} .
\]
Fix $T>0$ and set $t_*:=\min\{T,\bar t\,\}\in(0,T]$ with $\bar t$ as
in \eqref{eq:speeds}. Let $\kappa\in(0,\tfrac1{100})$ and let
$B_\kappa$, $A_\kappa=Z(t_*,B_\kappa)$ be as in Lemma
\ref{lem:transit}. For $y\in A_\kappa$ write $z:=Z(t_*,\cdot)^{-1}(y)
\in B_\kappa$; then by parts (a) and (c) of that lemma,
\[
\rho_{t_*}(y)=\frac{p(y)}{p(z)}
\ \ge\ \frac{c(t_*)}{\tfrac32\kappa}
\ \xrightarrow[\kappa\downarrow0]{}\ \infty ,
\]
while $\LL{2}(A_\kappa)>0$ by part (d). Hence for every constant $L$
there is $\kappa$ with $Z(t_*,\cdot)_\#\LL{2}\ge L\LL{2}$ on a set
of positive measure. Since Definition \ref{def:rlf}(ii) demands a
single constant $L$ valid for \emph{all} $t\in[0,T]$, and this fails
at $t=t_*\in(0,T]$, no regular Lagrangian flow exists on $[0,T]$.
\end{proof}

\begin{remark}
\label{rem:intermediate}
The orbits of (iii) are periodic in the co-moving frame, so a
trajectory started near the wall passes through the interior and may
return near the wall; nothing in the construction guarantees that
the compression witnessed above is still present at the exact
terminal time $T$, and for large $T$ a statement about
$Z(T,\cdot)_\#\LL{2}$ alone would require an analysis of the period
map. None is needed: failure of the compressibility bound at a
single time $t_*\in(0,T]$ already contradicts Definition
\ref{def:rlf}(ii), which quantifies over all $t\in[0,T]$. The
numerical experiment of Section \ref{subsec:numerics} measures
$\sup_{t\le T}\rho_t$, which is precisely the quantity appearing in
the corrected statement.
\end{remark}

\begin{remark}[A wall-localised compression mechanism]
\label{rem:coherence}
The infinite compression in part (iv) is a wall-localised
infinite-compression mechanism \emph{analogous} to the singular
boundary divergence sheet that Theorem \ref{thm:A} excludes through
tangency; the two are not identified by Lemma \ref{lem:sheet}, which
does not apply to this field, but by the exhaustion computation
above: on each surface $\{x_2=\eta\}$ the extension of $v$ carries a
nonzero singular divergence of size independent of $\eta$, and it is
along the orbits threading those surfaces that the Jacobian
degenerates. The example separates three layers that the
whole-space theory tends to fuse: existence of invariant
characteristics (true here, by the first integral, and guaranteed in
general by Proposition \ref{prop:conf}); pointwise deterministic
transport of the marginals (true here); and the regular Lagrangian
flow property (false here, for every $T>0$).
\end{remark}

\begin{remark}\label{rem:whatitshows}
Since $v\notin L^1$ on every collar, the example violates
\textup{(F2)} and \textup{(F3)} as well as \textup{(F4)}. It
therefore does \emph{not} show that tangency alone is
indispensable in Theorem \ref{thm:A}; what it shows is that the
collection of boundary hypotheses \textup{(F2)--(F4)} cannot be
relaxed far enough to admit the boundary-current mechanism. This is
not an accident of the construction. In the regime of Lemma
\ref{lem:taylor} the hypotheses are structurally entangled: by
Corollary \ref{cor:entangled} below, whenever the density vanishes
linearly at a $C^1$ wall and the flux is $C^1$ up to it, failure of
tangency \emph{forces} $v\notin L^1$ near that wall. Within this
class, therefore, no example can isolate \textup{(F4)}, and it is unclear whether tangency is independently necessary. %--- for fields with an $L^1$ collar bound and a degenerate or non-$C^1$ boundary vacuum --- is open.
\end{remark}

\subsection{Numerical illustration}\label{subsec:numerics}

Figure \ref{fig:bc} displays the mechanism, for $\lambda=1$ and
$T=0.5$. Panel (a) shows the orbit structure in the co-moving frame:
level sets of the first integral $\bar\psi=\sin(\pi x_2)\cos\xi$,
closed curves encircling the elliptic centres and hugging both walls,
with the separatrix $\{\bar\psi=0\}$ (walls and the circles
$\{\cos\xi=0\}$) dashed. Panel (b) is the trace of the normal
velocity on the bottom wall, $v_2(\xi,0)=4\lambda\sin\xi/(\pi
m(\xi))$: nonzero for a.e.\ $\xi$ and outward on the moving
half-wall, the failure of tangency in Theorem \ref{thm:B}(ii) made
visible. Panels (c) and (d) quantify part (iv). For each
$\kappa$, a trajectory is started at the bottom of the lower arc of
the orbit $\{\bar\psi=\kappa\}$ and integrated with an adaptive
high-order scheme; the drift $\sup_{t\le T}|\bar\psi(Z_t)-\kappa|$
remains below $10^{-11}$ throughout. The witnessed compression, measured as
$\sup_{t\le T}\rho_t$ in accordance with Theorem
\ref{thm:B}(iv) and Remark \ref{rem:intermediate},
$\sup_{t\le T}\rho_t$, with $\rho_t=p(Z_t)/p(Z_0)$, follows the
predicted law $c(T)/\kappa$ over two decades:
\[
\begin{array}{c|ccccc}
\kappa & 0.1 & 0.0316 & 0.01 & 0.0032 & 0.001\\ \hline
\sup_{t\le T}\rho_t & 7.00 & 21.4 & 67.0 & 211 & 667\\
\kappa\cdot\sup_{t\le T}\rho_t & 0.700 & 0.677 & 0.670 & 0.668 &
0.667
\end{array}
\]
so that $c(T)\approx\tfrac23$ at $T=0.5$: the ``elementary
phase-plane estimates'' invoked in the proof of Theorem
\ref{thm:B}(iv) are confirmed with the constant included. Panel (d)
shows the density $\rho_T(y)=p(y)/p(Z(T,\cdot)^{-1}(y))$ of the
push-forward $Z(T,\cdot)_\#\LL{2}$, computed pointwise from the
exact formula by backward integration on a grid: the compression
(red) concentrates in bands along the walls downstream of the
turning regions, the depletion (blue) in the regions the near-wall
mass has vacated --- the wall-localised compression mechanism
described in Remark \ref{rem:coherence}. A numerical illustration of the data-endpoint
example of Section \ref{subsec:dataex} would essentially reproduce
the corresponding figure of \cite{pfode}, the image corrections
being exponentially small at plotting scale, and is omitted.

\begin{figure}[t]
\centering
\includegraphics[width=\textwidth]{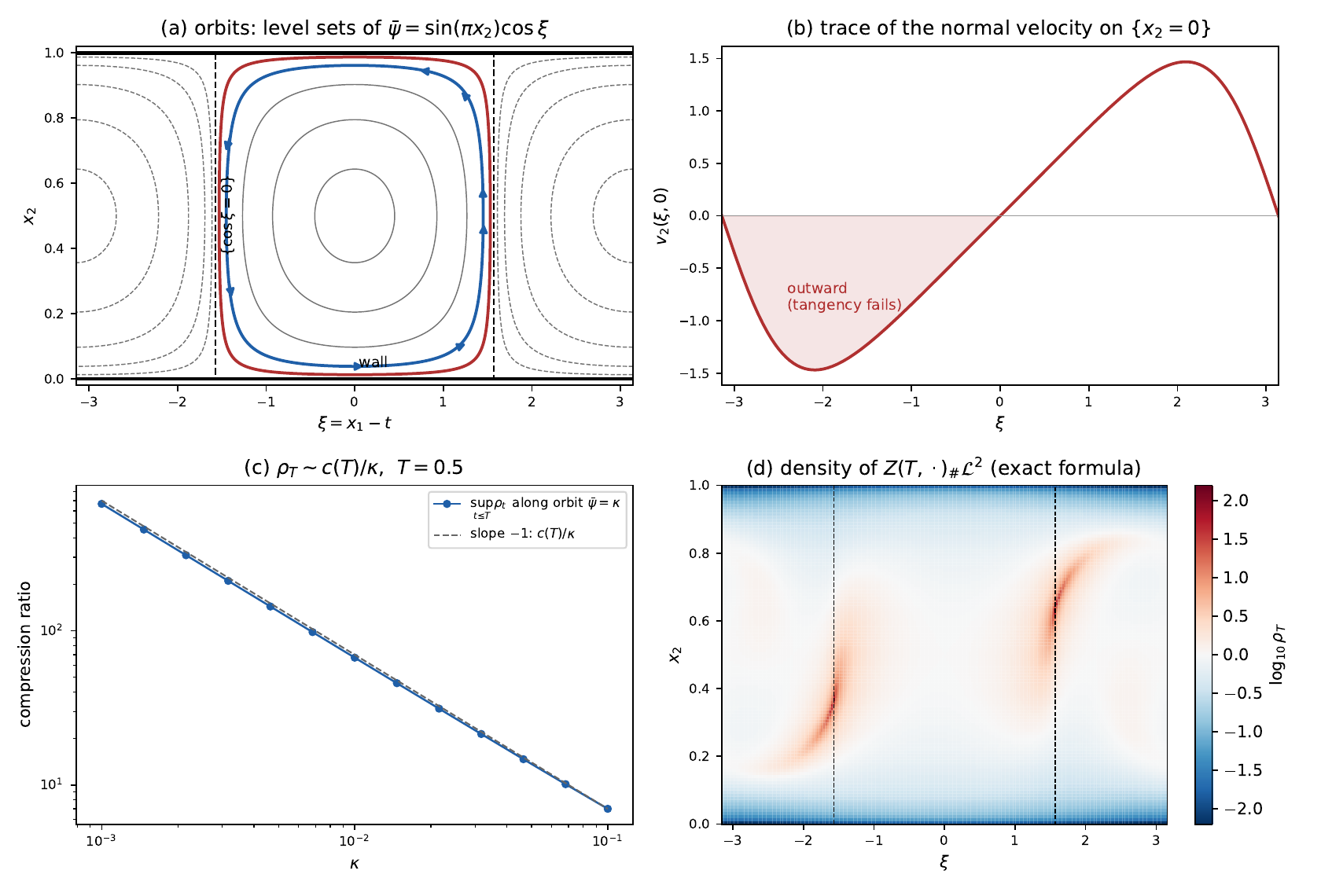}
\caption{The boundary-current example \eqref{eq:bce}, $\lambda=1$.
\textbf{(a)} Orbits of the probability-flow ODE in the co-moving
frame: level sets of $\bar\psi$; the two highlighted orbits
($\kappa=0.04$, $0.12$) hug both walls and pass through mid-height
at the turns; separatrix dashed; arrows indicate the direction of
motion. \textbf{(b)} Trace of the normal velocity on the bottom
wall: nonzero a.e., outward on half the wall (shaded) --- tangency
fails, yet the normal \emph{flux} $p\,v_2$ vanishes because $p$
does. \textbf{(c)} Witnessed compression along the orbit
$\{\bar\psi=\kappa\}$ up to time $T=0.5$, against the predicted
$c(T)/\kappa$ (slope $-1$): no uniform compressibility constant
exists. \textbf{(d)} Push-forward density $\rho_T$ of Lebesgue
measure under the flow, from the exact formula
$\rho_T(y)=p(y)/p(Z(T,\cdot)^{-1}(y))$: compression concentrates
along the walls downstream of the turning regions.}
\label{fig:bc}
\end{figure}

%=====================================================================
\section{Mimicking reflected diffusions with differential equations}
\label{sec:reflected}
%=====================================================================

\subsection{The reflection term disappears from the probability flow}

Consider the Skorokhod problem in a bounded $C^{1,1}$ domain,
\begin{equation}\label{eq:refSDE}
dX_t=f(X_t,t)\dd t+\sigma(t)\dd W_t-\nu(X_t)\dd L_t,\qquad
X_t\in\bar\Omega,\qquad X_0\sim\mu_0\in\PP(\bar\Omega),
\end{equation}
with $\sigma$ scalar, $0<\epsilon\le\sigma\le\epsilon^{-1}$, and $L$
the boundary local time. For bounded measurable $f$,
\eqref{eq:refSDE} is well posed in law --- by the Skorokhod-map
theory for normal reflection in smooth domains \cite{LionsSznitman},
extended to far more general reflection fields and domains by the
convex-duality and extended-Skorokhod-problem framework
\cite{DupuisRamanan,RamananESP} and the submartingale-problem
characterisation \cite{KangRamananSub}; 
for $t>0$ the marginal law
has a density $p_t$, H\"older continuous and strictly positive on
$\bar\Omega$ %--- De Giorgi--Nash--Moser theory and two-sided Gaussian bounds for Neumann-type kernels of divergence-form operators, cf.\ \cite{GSC} --- 
and solves the Fokker--Planck equation with co-normal
no-flux condition
$(\tfrac12\sigma^2\nabla p_t-fp_t)\cdot\nu=0$ on $\partial\Omega$.
Regularity beyond H\"older continuity requires regularity of the
coefficients and is quantified in Corollary \ref{cor:early} below.
Define the probability-flow velocity and flux
\begin{equation}\label{eq:pfvel}
v:=f-\tfrac12\sigma^2\nabla\log p_t,\qquad
J:=fp_t-\tfrac12\sigma^2\nabla p_t=p_tv .
\end{equation}
The interior equation is the continuity equation
$\partial_tp+\dive(pv)=0$, and the boundary condition reads exactly
\begin{equation}\label{eq:tangencyref}
p_t\,v\cdot\nu=0\ \text{ on }\partial\Omega,
\qquad\text{hence}\qquad
v\cdot\nu=0\ \text{ on }\partial\Omega\ \text{ for }t>0,
\end{equation}
whenever $p_t$ is positive and differentiable up to the boundary:
the normal component of
the score exactly cancels $f\cdot\nu$, and the boundary local time
--- a genuinely stochastic object --- is entirely absorbed into
$\nabla\log p_t$. The probability-flow ODE therefore carries no
local time term, in contrast with the reverse-time \emph{SDE} of a
reflected diffusion, which retains one \cite{Cattiaux88}. This is
what makes deterministic samplers for constrained diffusion models
\cite{LouErmon} conceivable. The pair \eqref{eq:pfvel} is a no-flux
weak solution in the sense of Definition \ref{def:nfweak}, and the
question whether the density flow $(p_t)$ may be 'mimicked' by   a deterministic confined
transport  is answered by Theorem \ref{thm:A}:
\begin{corollary}[Early-stopped reflected transport]\label{cor:early}
Let $\partial\Omega\in C^{2,\alpha}$, let $\sigma$ satisfy
$0<\epsilon\le\sigma\le\epsilon^{-1}$ with $\sigma^2\in
C^{\alpha/2}([0,T])$, and let the drift satisfy
\begin{equation}\label{eq:driftclass}
f\in C^{1+\alpha,(1+\alpha)/2}\big(\bar\Omega\times[0,T];\R^d\big),
\end{equation}
so that in particular $\dive f\in
C^{\alpha,\alpha/2}(\bar\Omega\times[0,T])$ and
$\sup_t\|f(\cdot,t)\|_{C^1(\bar\Omega)}<\infty$. No compatibility
condition on the initial law is imposed: the estimates below are
interior in time.
Let $X$ solve \eqref{eq:refSDE} with \emph{arbitrary} initial law
$\mu_0\in\PP(\bar\Omega)$, possibly singular. Then for every
$\delta\in(0,T)$ the pair \eqref{eq:pfvel} satisfies hypotheses
\textup{(F0)--(F4)} on $[\delta,T]$, and the probability-flow ODE
generates a regular Lagrangian flow on $\bar\Omega$ transporting
$p_\delta$ onto $p_t$ for every $t\in[\delta,T]$, with no reflection
mechanism; in fact $[\dive v]^\pm\in L^\infty$, so the two-sided
bound of Remark \ref{rem:timerev} holds as well. The constants
degenerate as $\delta\downarrow0$.
\end{corollary}

\begin{proof}
\emph{Step 1 (regularity and positivity).}  Under \eqref{eq:driftclass},
$\sigma^2\in C^{\alpha/2}$ and $\partial\Omega\in C^{2,\alpha}$, all
coefficients of the co-normal (oblique) problem are parabolically
H\"older continuous. Fix $\delta\in(0,T)$. On
$\bar\Omega\times[\delta/2,T]$ the density is a bounded weak
solution --- by De Giorgi--Nash--Moser theory and the Gaussian
bounds recalled above --- and boundary Schauder estimates for the
co-normal problem \cite{LSU} are \emph{interior in time}: they
bound $\|p\|_{C^{2+\alpha,1+\alpha/2}(\bar\Omega\times[\delta,T])}$
in terms of $\|p\|_{L^\infty(\bar\Omega\times[\delta/2,T])}$, the
coefficient norms, and $(\delta/2)^{-1}$, with no compatibility
condition at the left endpoint. Restricting to $[\delta,T]$ gives
$p\in C^{2+\alpha,1+\alpha/2}(\bar\Omega\times[\delta,T])$. This
route avoids treating $\delta$ as an initial parabolic time, and
hence avoids assuming the very boundary regularity of $p_\delta$
that is being established. Moreover $p>0$ on $\bar\Omega\times[\delta,T]$:
interior positivity follows from the Harnack inequality, and a zero
at a boundary point $(x_0,t_0)$ would force $\partial_\nu
p(x_0,t_0)<0$ by the Hopf lemma, contradicting the co-normal
condition $\tfrac12\sigma^2\partial_\nu p=(f\cdot\nu)\,p=0$ at that
point. Set $c_\delta:=\inf_{\bar\Omega\times[\delta,T]}p>0$.

\emph{Step 2 (F0)--(F1).} Re-basing at time $\delta$, the pair
$(p_t,J_t)_{t\in[\delta,T]}$ with initial condition $p_\delta\in
C^{2,\alpha}(\bar\Omega)\subset L^1\cap L^\infty(\Omega)$ is a
classical solution of the no-flux problem, hence a no-flux weak
solution in the sense of Definition \ref{def:nfweak} by the
divergence theorem; $t\mapsto p_t\Ld$ is narrowly continuous, and
$p\ge c_\delta>0$ gives (F1).

\emph{Step 3 (F2).} By Step 1 and $p\ge c_\delta$,
$\nabla\log p=\nabla p/p\in C^{1,\alpha}(\bar\Omega;\R^d)$ with
norms uniform on $[\delta,T]$; by \eqref{eq:driftclass},
$f(\cdot,t)\in C^1(\bar\Omega)\subset W^{1,\infty}(\Omega)$ with
uniformly bounded norms. Hence
\[
v=f-\tfrac12\sigma^2\nabla\log p\in
L^\infty\big([\delta,T];W^{1,\infty}(\Omega;\R^d)\big)\subset
L^1\big([\delta,T];BV(\Omega;\R^d)\big),
\]
measurability in $t$ following from that of $f$ and $\sigma$ and
continuity of $t\mapsto p_t$ in $C^2(\bar\Omega)$. Global BV on
$\Omega$ implies both the interior-local and the collar parts of
(F2).

\emph{Step 4 (F3).} For a.e.\ $t$,
$D\!\cdot v_t=\big(\dive f_t-\tfrac12\sigma^2\Delta\log
p_t\big)\Ld$ is absolutely continuous: $f_t\in C^1(\bar\Omega)$
makes $\dive f_t$ a bounded continuous function, and
$\Delta\log p=\Delta p/p-|\nabla p|^2/p^2\in C^{\alpha}(\bar\Omega)$
with bounds uniform on $[\delta,T]$ by Step 1. Hence
$[\dive v]^\pm\in L^\infty([\delta,T]\times\Omega)$, which is (F3)
with a two-sided bonus.

\emph{Step 5 (F4).} Dividing the co-normal condition by $p>0$
gives, pointwise on $\partial\Omega$,
$v\cdot\nu=f\cdot\nu-\tfrac12\sigma^2\partial_\nu\log p=0$; for a
field in $C(\bar\Omega)\cap W^{1,\infty}(\Omega)$ the classical
boundary values coincide with the interior BV trace, so (F4) holds.

\emph{Step 6.} Theorem \ref{thm:A}, applied on $[\delta,T]$ with
initial condition $p_\delta$, yields the regular Lagrangian flow and
the transport identity; the constants involve $c_\delta^{-1}$ and
the Schauder norms on $[\delta,T]$, which degenerate as
$\delta\downarrow0$.
\end{proof}

\begin{remark}[The role of the drift assumptions]\label{rem:driftrole}

In nondivergence form
the equation reads
\[
\partial_tp-\tfrac12\sigma^2\Delta p+f\cdot\nabla p+(\dive f)\,p=0,
\]
whose zeroth-order coefficient is $\dive f$. Schauder theory
requires this coefficient, and not merely $f$, to be parabolically
H\"older continuous, which is why \eqref{eq:driftclass} imposes
$f\in C^{1+\alpha,(1+\alpha)/2}$ rather than $f\in
C^{\alpha,\alpha/2}$. The latter
gives $\dive f$ bounded but not H\"older, and yields only
first-derivative H\"older regularity of $p$, insufficient for the
Hessian control needed in Step 4.

Corollary \ref{cor:early} may also be stated conditionally: whenever the
reflected problem is well posed and the probability-flow velocity
satisfies \textup{(F2)--(F3)} on $[\delta,T]$, Theorem \ref{thm:A}
applies, tangency being supplied by Corollary \ref{cor:auto} as soon
as $p$ is $C^1$ and positive up to the boundary; the drift class
\eqref{eq:driftclass} is one verifiable instance, and the rougher
set below a second. In the regularity class of Corollary
\ref{cor:early} the field $v$ is
spatially Lipschitz on $[\delta,T]$ and the flow is in fact
classical; the content of the corollary is that the \emph{initial
law} may be arbitrary --- early stopping regularises the score, not
the drift. The drift hypotheses cannot be dispensed with by
regularity of $p$ alone: hypotheses \textup{(F2)--(F3)} constrain
$v=f-\tfrac12\sigma^2\nabla\log p$ jointly, and for merely bounded
measurable $f$ neither BV regularity nor an absolutely continuous,
one-sidedly bounded divergence of $v$ is available, however smooth
the score part may be. This mirrors the whole-space Lagrangian
theorem \cite[Theorem~4.5]{pfode}, which likewise imposes joint
drift--score hypotheses. Beyond the Schauder class
\eqref{eq:driftclass}, in which $v$ is Lipschitz, the natural
rougher hypotheses are the bounded-domain transcription of
$(D_L)+(S)$ of \cite{pfode}: for the drift,
\[
f\in L^1\big([0,T];BV(\Omega;\R^d)\big),
\quad
D\!\cdot f_t\ \text{absolutely continuous},
\quad
[\dive f]^-\in L^1\big([0,T];L^\infty(\Omega)\big),
\]
and for the score, $\nabla\log p\in
L^1\big([\delta,T];BV(\Omega;\R^d)\big)$ with
$[\Delta\log p]^+\in L^1([\delta,T];L^\infty(\Omega))$, together
with tangency. These give \textup{(F2)--(F4)} directly and are
where Theorem \ref{thm:A} earns its generality; unlike the Schauder
class, however, the score conditions are then hypotheses rather than
consequences, since parabolic regularity theory does not reach them
from rough coefficients.
\end{remark}

\begin{remark}[Bounded measurable drift: superposition, not flow]
\label{rem:rough}
For merely bounded measurable $f$, the marginal density is H\"older
continuous and strictly positive on $\bar\Omega\times[\delta,T]$, as
recalled above, but need not be $C^1$, and no flow-level conclusion
is claimed. What survives is a characteristics-level statement, and
it must be rebased at a positive time: for a singular initial law
the flux need not be square integrable up to $t=0$. Fix
$\delta\in(0,T)$. On $[\delta,T]$ the density is bounded, bounded
away from $0$, and satisfies the parabolic energy estimate, so
$J=fp-\tfrac12\sigma^2\nabla p\in L^2([\delta,T]\times\Omega)\subset
L^1([\delta,T]\times\Omega)$, and $p>0$ there makes the vacuum
hypothesis \eqref{eq:novacuumflux} vacuous. Proposition
\ref{prop:conf}, applied to the evolution rebased at $p_\delta$,
therefore yields a superposition of integral curves of $v$ confined
to $\bar\Omega$ on $[\delta,T]$ with marginals $p_t$. Whether a
regular Lagrangian flow exists in this generality, and what happens
on $(0,\delta)$, are open.
\end{remark}
The following corollary treats the case of two widely used diffusion models \cite{Song2021}.
\begin{corollary}[Reflected variance-exploding and variance-preserving
models]\label{cor:vpve}
Let $\partial\Omega\in C^{2,\alpha}$ and let $X$ solve
\eqref{eq:refSDE} with either
\begin{itemize}
\item[\textup{(VE)}] $f\equiv0$ and $0<\epsilon\le\sigma(t)\le
\epsilon^{-1}$ measurable, or
\item[\textup{(VP)}] $f(x,t)=-\beta(t)x$ and $\sigma(t)=
\sqrt{2\beta(t)}$ with $0<\beta_{\min}\le\beta(t)\le\beta_{\max}$
and $\beta\in C^{(1+\alpha)/2}([0,T])$ --- satisfied by the
schedules used in practice, which are smooth,
\end{itemize}
the reflected analogues of the forward diffusions used in
score-based generative models. Then for \emph{arbitrary}, possibly
singular, initial laws $\mu_0\in\PP(\bar\Omega)$ and every
$\delta\in(0,T)$, the probability-flow velocity satisfies
\textup{(F0)--(F4)} on $[\delta,T]$ and the probability-flow ODE
generates a regular Lagrangian flow on $\bar\Omega$ transporting
$p_\delta$ onto $p_t$, $t\in[\delta,T]$.
\end{corollary}

\begin{proof}
\emph{(VP).} The drift is spatially affine, so $\nabla f=-\beta(t)I$
and $\dive f=-d\,\beta(t)$ are as regular in time as $\beta$ and
constant in space; with $\beta\in C^{(1+\alpha)/2}([0,T])$ one has
$f\in C^{1+\alpha,(1+\alpha)/2}(\bar\Omega\times[0,T];\R^d)$, which
is \eqref{eq:driftclass}, and
$\sup_t\|f(\cdot,t)\|_{C^1(\bar\Omega)}\le
\beta_{\max}\big(1+\sup_{\bar\Omega}|x|\big)<\infty$ --- boundedness
of the domain doing what linear-growth conditions do in the whole
space. Moreover $\sigma^2=2\beta\in C^{\alpha/2}([0,T])$ with
$2\beta_{\min}\le\sigma^2\le2\beta_{\max}$; the time-H\"older
hypothesis on $\beta$ is the analogue of the time-regularity imposed
in \cite{pfode} for parabolic regularity of the score. The affine
structure is what makes this case safe: $\dive f$ is a function of
$t$ alone, so the zeroth-order coefficient of the nondivergence form
is automatically H\"older. All hypotheses of Corollary
\ref{cor:early} hold, and its conclusion is the claim. Note that $f\cdot\nu=-\beta\,x\cdot\nu\not\equiv0$, so
the boundary condition is genuinely co-normal (Remark
\ref{rem:vpve}) and Step 5 of the proof of Corollary
\ref{cor:early} uses it in full.

\emph{(VE).} Here $\sigma$ is merely measurable and the Schauder
input of Corollary \ref{cor:early} is not directly available; a
deterministic time change supplies it. Let
$\tau(t):=\int_0^t\tfrac12\sigma(s)^2\dd s$, a bi-Lipschitz
increasing bijection of $[0,T]$ onto $[0,\tau(T)]$ with
$\tfrac12\epsilon^2\le\tau'\le\tfrac12\epsilon^{-2}$. Then
$\tilde p_\tau:=p_{t(\tau)}$ solves the Neumann heat equation
$\partial_\tau\tilde p=\Delta\tilde p$, and Step 1 of the proof of
Corollary \ref{cor:early} (with $f\equiv0$ and unit diffusion)
gives $\tilde p\in
C^{2+\alpha,1+\alpha/2}\big(\bar\Omega\times[\tau(\delta),\tau(T)]\big)$
with $\inf\tilde p>0$. Back in the original time variable,
$\sup_{t\in[\delta,T]}\|p_t\|_{C^{2,\alpha}(\bar\Omega)}<\infty$,
$\inf_{\bar\Omega\times[\delta,T]}p>0$, and $t\mapsto p_t$ is
continuous into $C^2(\bar\Omega)$. Steps 2--5 of that proof now
apply verbatim in the original time variable:
$v=-\tfrac12\sigma(t)^2\nabla\log p_t\in
L^\infty([\delta,T];W^{1,\infty}(\Omega;\R^d))$, the bounded
measurable factor $\sigma^2$ being harmless;
$D\!\cdot v_t=-\tfrac12\sigma^2\Delta\log p_t\,\Ld$ is absolutely
continuous with a two-sided $L^\infty$ bound; and the Neumann
condition $\partial_\nu p=0$ gives $v\cdot\nu=0$ pointwise. Theorem
\ref{thm:A} concludes.
\end{proof}

\begin{remark}[No-flux versus Neumann; loss of Gaussian structure]
\label{rem:vpve}
Two features distinguish the reflected models from their whole-space
counterparts. First, for (VE) the no-flux condition reduces to the
homogeneous Neumann condition $\partial_\nu p=0$, but for (VP) it is
genuinely co-normal: $\partial_\nu p=(2f\cdot\nu/\sigma^2)\,p$ with
$f\cdot\nu=-\beta\,x\cdot\nu\not\equiv0$, a drift-dependent
Robin-type relation; imposing plain Neumann instead would violate
conservation of mass. Second, reflection destroys the
variation-of-constants representation: the marginals of reflected
Ornstein--Uhlenbeck dynamics are not Gaussian convolutions of the
initial law, so the explicit score bounds available in the whole
space have no direct analogue, and the regularity input for Corollary
\ref{cor:vpve} comes instead from Neumann parabolic theory
\cite{LSU}. The example of Section \ref{subsec:dataex} is exactly the
reflected (VE) model with $\sigma\equiv1$ started from compactly
supported data, and shows the early stopping in Corollary
\ref{cor:vpve} is necessary.
\end{remark}

\begin{remark}[Relation to the Skorokhod-problem and
probabilistic literature]\label{rem:ramanan}
The well-posedness statement above and the formulation (NF) relate to a substantial probabilistic literature on reflected diffusions. First, the
Skorokhod-map theory \cite{burdzy2004,LionsSznitman,saisho1987,DupuisRamanan,RamananESP}
achieves \emph{pathwise} stability of the reflected SDE through the
constraining mechanism, whereas Theorem \ref{thm:A} achieves
flow-level well-posedness of the ODE with the constraining mechanism
absorbed entirely into the score: two opposite resolutions of the
same confinement problem. Second, (NF) is the time-dependent,
normal-reflection, absolutely continuous case of the \emph{basic
adjoint relationship} used to characterise stationary distributions
of reflected diffusions \cite{HarrisonWilliams,KangRamananStat},
which in general involves a pair of measures --- an interior measure
and a boundary measure, the Eulerian shadow of the local time. For
normal reflection with nondegenerate diffusion and a density
positive up to the boundary, the boundary measure is determined by
the trace of $p$ and the pair collapses to a single density with the
co-normal condition: the Eulerian counterpart of the observation
\eqref{eq:tangencyref} that the local time disappears from the
probability-flow velocity. For \emph{oblique} reflection the
tangential component of the reflection field generically induces a
surface flux along $\partial\Omega$, and in the piecewise-smooth and
degenerate settings of \cite{KangRamananSub,KangRamananStat} the
marginals may charge the boundary, so the measure-pair formulation
is unavoidable --- a singular cousin of the boundary current of
Theorem \ref{thm:B}. Finally, in
nonsmooth domains the constraining term may fail to have bounded
variation and the reflected process is only a Dirichlet process
\cite{KangRamananDir}: the pathwise counterpart of the degradation
of score regularity at non-convex boundaries discussed in Remark
\ref{rem:convex} below.
\end{remark}

\begin{remark}[Geometry influences the score]\label{rem:convex}
In the classical regime more is true: for reflected Brownian motion
or reflected Ornstein--Uhlenbeck dynamics in a bounded \emph{convex}
$C^{1,1}$ domain, with $0<c_0\le p_0\le C_0$, $\log p_0\in
C^{1,1}(\bar\Omega)$ and the Neumann compatibility
$\nabla\log p_0\cdot\nu=2f\cdot\nu/\sigma^2$ on $\partial\Omega$,
log-Hessian bounds propagate along the Neumann semigroup ---
convexity making the boundary contribution to the Bakry--\'Emery
curvature nonnegative \cite{Wang} --- and the score is Lipschitz up
to $\bar\Omega$ uniformly on $[0,T]$. Convexity is not decorative:
at a re-entrant corner, second-derivative estimates for the Neumann
problem degrade and $\nabla^2\log p_t$ may blow up at the boundary
even for smooth interior data. The geometry of $\partial\Omega$ thus
enters the regularity theory of the score, a phenomenon absent from
the whole-space setting \cite{pfode}.
\end{remark}

\subsection{Failure at the endpoint: the boundary is innocent}
\label{subsec:dataex}

The early-stopping in Corollary \ref{cor:early} is necessary: we now
show that on the full interval $[0,T]$, Eulerian well-posedness does
not imply deterministic transport, through a mechanism inherited from
the data --- and that the reflecting boundary is provably innocent.

Consider reflected Brownian motion in $\bar\Omega=[-2,2]$:
\begin{equation}\label{eq:dataex}
d=1,\qquad f\equiv0,\qquad\sigma\equiv1,\qquad
p_0=\tfrac12\mathbf{1}_{[-1,1]} .
\end{equation}
With $f\equiv0$, Proposition \ref{prop:energyuniq} applies trivially
and gives uniqueness in $\mathcal{X}_\Omega$ outright. The marginal
flow is given by the method of images: with $\varphi_t$ the heat
kernel, $q_t:=p_0*\varphi_t$ the whole-line solution, and the
$8$-periodic even extension of $p_0$ across $x=\pm2$,
\begin{equation}\label{eq:images}
p_t(x)=\sum_{n\in\mathbb{Z}}\Big[q_t(x-8n)+q_t(4-x-8n)\Big],
\qquad x\in[-2,2],\ t>0,
\end{equation}
smooth, strictly positive on $[-2,2]$ for $t>0$, with
$\partial_xp_t(\pm2)=0$ and $\|\partial_xp_t\|_{L^2}^2=O(t^{-1/2})$,
hence in $\mathcal{X}_\Omega$: it is \emph{the} no-flux density
evolution for \eqref{eq:dataex}. Note $v(\pm2,t)=
-\tfrac12\partial_x\log p_t(\pm2)=0$ for every $t>0$: tangency holds
at the walls at all positive times.

\begin{lemma}[Edge asymptotics]\label{lem:asymp}
Fix a compact $K\subset(1,2)$. For $x\in K$, as $t\downarrow0$,
\[
\log p_t(x)=-\frac{(x-1)^2}{2t}-\log(x-1)+\tfrac12\log t+c+O(t),
\qquad
v(x,t)=\frac{x-1}{2t}+O(1),
\]
uniformly on $K$, and symmetrically on compacts of $(-2,-1)$.
\end{lemma}

\begin{proof}
For $x>1$, $q_t(x)=\tfrac14\big[\operatorname{erfc}\big(\tfrac{x-1}
{\sqrt{2t}}\big)-\operatorname{erfc}\big(\tfrac{x+1}{\sqrt{2t}}\big)
\big]$; inserting
$\operatorname{erfc}(u)=\tfrac{e^{-u^2}}{u\sqrt{\pi}}\big(1-
\tfrac1{2u^2}+O(u^{-4})\big)$ with $u=\tfrac{x-1}{\sqrt{2t}}$ gives
the expansion of $\log q_t$; the second term is smaller by a factor
$e^{-2x/t}$. In \eqref{eq:images} the term $q_t(x)$ dominates: the
nearest competitor $q_t(4-x)$ has support-edge distance $3-x$, and
$(3-x)^2-(x-1)^2=8-4x\ge8-4\max K>0$, so its ratio to $q_t(x)$ is
$O(e^{-(8-4x)/2t})$; remaining images are farther by at least
distance $3$. Differentiating in $x$, the exponentially small
corrections are absorbed in the errors, and
$v=-\tfrac12\partial_x\log p_t$.
\end{proof}

\begin{theorem}[Failure of transport at the data endpoint]
\label{thm:datafail}
Let $v=-\tfrac12\partial_x\log p_t$ with $p_t$ as in
\eqref{eq:images}, and let $V:=\{1<|x|<2\}$.
\begin{itemize}
\item[(i)] For every compact $K\subset V$,
$\int_0^T\!\int_K|v|\dd x\dd t=+\infty$: the interior part of
\textup{(F2)} fails. The collar conditions fail as well: the $1/t$
blow-up persists up to the walls, and near $x=\pm2$ one has
$[\partial_x^2\log p_t]^+\asymp t^{-2}$, so \textup{(F3)} also fails
there as $t\downarrow0$. Hypotheses \textup{(F0)--(F1)} hold, and
tangency \textup{(F4)} holds for every $t>0$.
\item[(ii)] For every $x_0\in V$ there is no absolutely continuous
$Z:[0,T]\to[-2,2]$ with $Z(0)=x_0$ and
$Z(t)=x_0+\int_0^tv(Z(s),s)\dd s$. Since $\LL{1}(V)>0$, no regular
Lagrangian flow on $\bar\Omega$ exists from time $0$.
\item[(iii)] By contrast, $Z\equiv2$ and $Z\equiv-2$ \emph{are}
integral curves from the boundary points; for every $x_0\in(-1,1)$
the quantile flow $Z(t,x_0)=F_t^{-1}(F_0(x_0))$ is the unique
integral curve, with $\partial_xZ(t,x)=p_0(x)/p_t(Z(t,x))$, and the
transport identity $Z(t,\cdot)_\#(p_0\LL{1})=p_t\LL{1}$ holds.
\item[(iv)] On $[\delta,T]$, \ref{cor:early} applies, with
constants of order $\delta^{-1}$.
\end{itemize}
\end{theorem}

\begin{proof}
(i) is Lemma \ref{lem:asymp} together with
$\int_0^{t_0}\!dt/t=\infty$; positivity of $p_t$ gives (F1); the
claim on $[\partial_x^2\log p_t]^+$ near the walls follows from the
fact that at $x=2$ the two dominant images make $p_t$ even about $2$
with a local minimum there, and $\partial_x^2\log
p_t(2)=p_t''(2)/p_t(2)=q_t''(2)/q_t(2)\,(1+o(1))\asymp t^{-2}$;
tangency at the walls is the Neumann symmetry noted above.

(ii) Set $F_t(x):=\int_{-2}^xp_t$. The flux vanishes at $x=-2$, so
Lemma \ref{lem:quantile} applies on $(-2,2)$, and $t\mapsto
F_t(Z(t))$ is constant along any integral curve on $(0,T]$, where
$v$ is smooth. The new feature of the bounded domain: $F_t:[-2,2]\to
[0,1]$ is \emph{onto}, and for $t>0$ strictly increasing, so the
value $1$ \emph{is} attained --- exactly at $x=2$. Let $Z$ be an
integral curve with $Z(0)=x_0\in(1,2)$. Since
$\|F_t-F_0\|_\infty\le\|p_t-p_0\|_{L^1}\to0$ and $Z$ is continuous
at $0$,
\[
F_t(Z(t))\equiv q:=\lim_{t\downarrow0}F_t(Z(t))=F_0(x_0)=1 .
\]
For $t>0$, $F_t(Z(t))=1$ forces $Z(t)=2$; continuity at $t=0$ then
gives $x_0=2$, a contradiction: an integral curve from the vacuum
would have to jump instantaneously to the reflecting wall. The case
$x_0\in(-2,-1)$ is symmetric.

(iii) $Z\equiv2$ satisfies the integral equation since $v(2,t)=0$
for $t>0$ and $\int_0^t|v(2,s)|\dd s=0$. For $x_0\in(-1,1)$:
$q:=F_0(x_0)\in(0,1)$, $Z(t,x_0):=F_t^{-1}(q)$ is well defined,
$\dot Z=v(Z,t)$ by Lemma \ref{lem:quantile}, $Z(t,x_0)\to x_0$ as
$t\downarrow0$ by uniform convergence of $F_t$ and strict
monotonicity, uniqueness holds because $v$ is locally Lipschitz on
$(-2,2)\times(0,T]$ and the first integral pins the curve;
differentiating $F_t(Z)=F_0(x)$ gives the Jacobian identity and the
transport identity. The image corrections in \eqref{eq:images} are
$O(e^{-c/t})$ on the relevant compacts and affect no step.

(iv) is \ref{cor:early}; the Tweedie identity applied to
the periodised initial law gives $|v|\le C/\delta$ and
$\|\partial_xv\|_\infty\le C/\delta^2$ on $[\delta,T]$.
\end{proof}

\begin{remark}[The boundary is innocent]\label{rem:innocent}
The only genuinely boundary-related hypothesis of Theorem
\ref{thm:A} --- tangency --- holds here at every positive time, and
the boundary points carry honest integral curves. What fails is
interior time-integrability, through the $1/t$ blow-up of the score
in the vacuum: the same data mechanism as in the whole-space
counterexample of \cite{pfode}. The comparison is instructive: on
$\R$, trajectories from the vacuum are excluded because the extreme
quantile level is never attained; on $[-2,2]$ it \emph{is} attained,
at the wall, and trajectories are excluded because reaching it would
violate continuity. Reflection changes the mechanism of the
contradiction but not the conclusion.
\end{remark}

%=====================================================================
\section{Rigidity: constraints on boundary mechanisms}
\label{sec:rigidity}
%=====================================================================

Three rigidity results delimit the possible failures of deterministic
confined transport, and show that under their hypotheses the
structure realised in Theorem \ref{thm:B} is forced. They do not
amount to a classification: they exclude the failure of invariance
in one dimension, the absence of confined characteristics whenever
the flux is integrable, and the failure of tangency wherever the
velocity is bounded with a nondegenerate boundary vacuum. Outside
these hypotheses --- in particular for low-regularity fields in
higher dimension --- other boundary mechanisms are not excluded.
%see Open problem \ref{op:invariance}.

\begin{proposition}[One-dimensional rigidity]\label{prop:1d}
Let $\Omega=(a,b)$ and let $(p_t)$ be a no-flux density evolution
with $p_t>0$ a.e.\ in $\Omega$ for \emph{every} $t\in[0,T]$,
$p_tv\in L^1_{\mathrm{loc}}$, and $t\mapsto p_t\LL{1}$ narrowly
continuous. Then the quantile flow $Z(t,x):=F_t^{-1}(F_0(x))$ is
well defined for a.e.\ $x$, takes values in $[a,b]$, transports
$p_0\LL{1}$ onto $p_t\LL{1}$, and its trajectories are integral
curves of $v$ wherever Lemma \ref{lem:quantile} applies. Invariance
can never fail in $d=1$: mass conservation and no-flux pin
$F_t(a)=0$, $F_t(b)=1$, and every quantile level $q\in(0,1)$ is
attained inside $[a,b]$ at every time. The only Lagrangian
pathologies available in one dimension are those of Theorem
\ref{thm:datafail} --- non-existence of curves from an initial
vacuum, or unbounded compression where $p_t$ degenerates --- never
an exit through the boundary.
\end{proposition}

\begin{proof}
$F_t:[a,b]\to[0,1]$ is continuous, nondecreasing, onto by mass
conservation and no-flux, and strictly increasing since $p_t>0$
a.e.; the rest follows from Lemma \ref{lem:quantile}, whose
flux-vanishing hypothesis at $a$ is the no-flux condition, and from
narrow continuity.
\end{proof}

\begin{proposition}[Superposition confinement]\label{prop:conf}
Let $\Omega$ be a bounded Lipschitz domain and $(p,J)$ a no-flux
weak solution with $J\in L^1([0,T]\times\Omega)$, and assume that
the flux vanishes on the vacuum,
\begin{equation}\label{eq:novacuumflux}
J=0 \quad \LL{d+1}\text{-a.e.\ on }\{p=0\},
\end{equation}
setting $v:=J/p$ there $($say $v:=0)$; \eqref{eq:novacuumflux} holds
in particular whenever $p_t>0$ a.e.\ for a.e.\ $t$, as under
\textup{(F1)}. Then there exists
$\eta\in\PP\big(C([0,T];\R^d)\big)$ concentrated on integral curves
of $v$ which remain in $\bar\Omega$ for all $t$, with
$(e_t)_\#\eta=p_t\Ld$ for every $t$.
\end{proposition}

\begin{proof}
As in Step 3 of the proof of Theorem \ref{thm:A}, the zero
extensions solve the continuity equation on $\R^d$. Hypothesis
\eqref{eq:novacuumflux} gives $p|v|=|J|$ $\LL{d+1}$-a.e.\ --- on
$\{p>0\}$ by definition of $v$, and on $\{p=0\}$ because both sides
vanish --- so the integrability hypothesis of Theorem
\ref{thm:superposition} reads
$\int_0^T\!\int\frac{|v|}{1+|x|}p\dd x\dd t\le
\int_0^T\!\int_\Omega|J|<\infty$. The superposition $\eta$ has
time-marginals carried by $\bar\Omega$, and curves being continuous,
$\eta$-a.e.\ curve lies in $\bar\Omega$ for all rational, hence
all, times.
\end{proof}

\begin{remark}\label{rem:vacuumflux}
Hypothesis \eqref{eq:novacuumflux} cannot be omitted: Definition
\ref{def:nfweak} permits $p_t$ to vanish on a set of positive
measure, where $v=J/p$ is undefined and the identity $p|v|=|J|$
--- the only point at which $J\in L^1$ is converted into the
weighted bound required by Theorem \ref{thm:superposition} ---
has no meaning. In all applications made here the hypothesis is
automatic, the density being strictly positive in the interior.
\end{remark}

Thus, whenever the total flux is integrable, invariant
characteristics exist from $p_0$-a.e.\ starting point: a
counterexample can only target the \emph{flow} structure ---
uniqueness, measurable selection, or compressibility --- never the
existence of confined integral curves. This is precisely the
situation of Theorem \ref{thm:B}.

\begin{lemma}[Taylor rigidity at a boundary vacuum]\label{lem:taylor}
Let $\Gamma\subset\partial\Omega$ be a relatively open $C^1$
boundary patch, $U$ a neighbourhood of $\Gamma$, and suppose that on
$(\Omega\cap U)\times(t_1,t_2)$: $J\in C^1$ up to $\Gamma$; $p\in
C^1$ up to $\Gamma$ with $p|_\Gamma\equiv0$ and
$p(x)\ge c\operatorname{dist}(x,\Gamma)$; and $v=J/p\in L^\infty$.
If the no-flux condition holds on $\Gamma$, then
$v(x,t)\cdot\nu\to0$ as $x\to\Gamma$, locally uniformly: tangency
cannot fail where the velocity is bounded.
\end{lemma}

\begin{proof}
On $\Gamma$: $J\cdot\nu=0$ (no-flux) and
$|J_\tau|=p|v_\tau|\le\|v\|_\infty p\to0$, so $J|_\Gamma=0$; since
$J\in C^1$ and $J_\tau$ vanishes identically on $\Gamma$, all
tangential derivatives of $J_\tau$ vanish there, i.e.\
$\nabla_\tau\!\cdot J_\tau|_\Gamma=0$. Moreover $p|_\Gamma\equiv0$
on a time interval gives $\partial_tp|_\Gamma=0$, and the continuity
equation yields
$\partial_\nu(J\cdot\nu)|_\Gamma=-\partial_tp|_\Gamma-
\nabla_\tau\!\cdot J_\tau|_\Gamma=0$. Hence
$J\cdot\nu=o(\operatorname{dist}(x,\Gamma))$, and the nondegeneracy
of $p$ gives $v\cdot\nu\to0$.
\end{proof}

Lemma \ref{lem:taylor} forces the structure \eqref{eq:bcstructure} of
any boundary-caused example: the tangential flux must \emph{not}
vanish at the wall. The price is unavoidable, and quantifies the
entanglement of the boundary hypotheses of Theorem \ref{thm:A}.

\begin{corollary}[Failure of tangency forces non-integrability]
\label{cor:entangled}
Let $\Gamma\subset\partial\Omega$ be a relatively open $C^{1,1}$ patch,
let $p$ and $J$ be of class $C^1$ in space and time up to $\Gamma$
on a time interval, with $J\cdot\nu=0$ on $\Gamma$, and suppose $p$
vanishes on $\Gamma$ nondegenerately:
$p(x)\ge c\operatorname{dist}(x,\Gamma)$ with $c>0$. Let
$x_0\in\Gamma$ and $\nu_0:=\nu(x_0)$.
Fix $t_0$ 
 in the stated time interval. If
\[
v(x,t_0)\cdot\nu_0\ \not\longrightarrow\ 0
\qquad\text{as }\Omega\ni x\to x_0
\]
\emph{(}the approach being unrestricted\emph{)}, then $v(.,t_0)\notin L^1$
on any spatial collar of $x_0$.
\end{corollary}

\begin{proof}
We argue invariantly. Fix $x_0\in\Gamma$, write $\nu_0:=\nu(x_0)$ and
let $\{\tau_1,\dots,\tau_{d-1}\}$ be an orthonormal basis of the
tangent space $\nu_0^\perp$; only the \emph{single} vector $\nu_0$
is ever differentiated against.
No derivative of the normal field is required; the $C^{1,1}$
 assumption is used only to obtain the tubular nearest-point projection.
The curved boundary enters only through tangential differentiation
of functions defined on $\bar\Omega$, which for a $C^1$ patch is
well defined without any chart: for $\tau\in\nu_0^\perp$ we set
$\partial_\tau w(x_0):=(w\circ\gamma)'(0)$ for any $C^1$ curve
$\gamma$ in $\Gamma$ with $\gamma(0)=x_0$, $\gamma'(0)=\tau$, which
agrees with the ambient derivative $\tau\cdot\nabla w(x_0)$ whenever
$w$ is $C^1$ up to $\Gamma$ and is therefore independent of the
curve. This is the only property of $\Gamma$ used below, and it is
why no flattening --- which would not preserve the divergence, and
would require second-order regularity of $\partial\Omega$ to
control the error --- is needed.

We also record, for later use, that the $C^1$ regularity of $p$
together with $p|_\Gamma\equiv0$ gives the \emph{upper} bound
$p(x)\le\|\nabla p\|_\infty\operatorname{dist}(x,\partial\Omega)$
near $x_0$, so that the nondegeneracy hypothesis makes
$p\asymp\operatorname{dist}$ there.

\emph{Step 1: $\dive J=0$ on $\Gamma$.} Since $p|_\Gamma\equiv0$
throughout a time interval, $\partial_tp=0$ on $\Gamma$; the
continuity equation holds in $\Omega$ and both of its terms extend
continuously to $\Gamma$, so $\dive J=-\partial_tp=0$ there.

\emph{Step 2:} One of the following holds:
\begin{itemize}
\item[(A)] $J$ vanishes identically on some relatively open
$\Gamma'\subset\Gamma$ containing $x_0$; or
\item[(B)] every relative neighbourhood of $x_0$ in $\Gamma$
contains a point at which $J\neq0$.
\end{itemize}
We show that \textup{(A)} implies $v(x)\cdot\nu_0\to0$ as
$\Omega\ni x\to x_0$, and that \textup{(B)} implies $v\notin L^1$ on
every collar of $x_0$. The corollary follows: if the limit fails,
\textup{(A)} is excluded by the first implication, so \textup{(B)}
holds and the second applies. Arguing through the dichotomy, rather
than through the value of a directional limit, is what allows the
hypothesis to be the unrestricted limit; a computation along the
inward normal ray alone would prove only the weaker statement in
which that particular limit is assumed nonzero.

\emph{Step 3: \textup{(A)} implies tangency at $x_0$.} Let $J\equiv0$
on $\Gamma'$. For $x'\in\Gamma'$ and $\tau$ tangent at $x'$,
differentiating along a $C^1$ curve in $\Gamma'$ gives
$\partial_\tau J(x')=0$; hence, with an orthonormal tangent frame
$\{\tau_i'\}$ at $x'$ and $\nu':=\nu(x')$, Step 1 yields
\begin{equation}\label{eq:entangledlimit}
\nu'^{\top}(\nabla J)(x')\,\nu'
=\operatorname{tr}(\nabla J)(x')-\sum_{i=1}^{d-1}
\tau_i'\cdot\big(\partial_{\tau_i'}J\big)(x')=0 ,
\end{equation}
using $\tau^{\top}(\nabla J)\tau=\tau\cdot\partial_\tau J$. No
curvature term appears here: the second fundamental form would enter
multiplied by $J\cdot\nu$, which vanishes on $\Gamma$. Now let $\pi$
be a nearest-point projection onto $\partial\Omega$, defined for $x$
near $x_0$, so that $\pi(x)\in\Gamma'$ and
$x-\pi(x)=-d(x)\,\nu(\pi(x))$ with
$d(x):=\operatorname{dist}(x,\partial\Omega)$. Since $\nabla J$ is
uniformly continuous near $x_0$ and $J(\pi(x))=0$,
\[
J(x)=-d(x)\,(\nabla J)(\pi(x))\,\nu(\pi(x))+o\big(d(x)\big),
\]
hence
\[
J(x)\cdot\nu_0=-d(x)\,\nu_0^{\top}(\nabla J)(\pi(x))\,\nu(\pi(x))
+o\big(d(x)\big).
\]
As $x\to x_0$ we have $\pi(x)\to x_0$ and, $\Gamma$ being $C^1$,
$\nu(\pi(x))\to\nu_0$; by continuity of $\nabla J$ and
\eqref{eq:entangledlimit} at $x'=x_0$, the bracketed factor tends to
$\nu_0^{\top}(\nabla J)(x_0)\nu_0=0$. Therefore $J(x)\cdot\nu_0
=o(d(x))$, and the nondegeneracy $p\ge c\,d$ gives
\[
\big|v(x)\cdot\nu_0\big|=\frac{|J(x)\cdot\nu_0|}{p(x)}
\le\frac{o(d(x))}{c\,d(x)}\longrightarrow0 ,
\]
the approach being unrestricted.

\emph{Step 4: \textup{(B)} implies non-integrability.} Fix any collar
of $x_0$ and choose, inside it, a point $x_1\in\Gamma$ with
$J(x_1)\neq0$. By continuity there are $\eta>0$ and $r>0$ with
$|J|\ge\eta$ on $B_r(x_1)\cap\Omega$, and by the upper bound
recorded above $p\le\|\nabla p\|_\infty d$ there, whence
\[
|v|=\frac{|J|}{p}\ \ge\ \frac{\eta}{\|\nabla p\|_\infty\,d(\cdot)}
\qquad\text{on }B_r(x_1)\cap\Omega .
\]
Integrating in the normal direction, $\int_0^r s^{-1}\dd s=\infty$,
so $v\notin L^1(B_r(x_1)\cap\Omega)$ and a fortiori $v\notin L^1$ on
the collar. As the collar was arbitrary, this holds for every collar
of $x_0$.

\end{proof}

Thus, in the class covered by Lemma \ref{lem:taylor}, tangency and
the collar integrability of $v$ cannot fail independently: the
mechanism that destroys one destroys the other. This is consistent
with Remark \ref{rem:sharp}(a) and delimits what the example of
Section \ref{sec:mainB} can prove; see Remark
\ref{rem:whatitshows}.

\begin{remark}[Two independent mechanisms]\label{rem:mechanisms}
Theorem \ref{thm:datafail} exhibits the \emph{data mechanism}: a
vacuum in $p_0$ makes integral curves fail to exist, at $t=0$ only,
in any dimension including $d=1$. Theorem \ref{thm:B} exhibits the
\emph{boundary mechanism}: a boundary current makes the flow
property fail on every time interval --- at some time $t_*\in(0,T]$
for each $T>0$, hence arbitrarily close to time zero, though not
necessarily at each prescribed time, the orbits being periodic ---
with all integral curves present --- and Proposition \ref{prop:1d} shows this mechanism is
genuinely higher-dimensional, since the tangential direction along
the wall is what carries the current. The discrepancy between the
density evolution and its deterministic transport representation,
identified in \cite{pfode}, therefore has (at least) two independent
sources, of which exactly one survives in dimension one.
\end{remark}
\newpage
%=====================================================================
\section{Discussion and extensions}\label{sec:open}
%=====================================================================
Circling back to the reflected diffusion models discussed in Section \ref{sec:generative},
we can observe that our results 
provide a rigorous mathematical justification for using the reflection-free probability-flow ODE \eqref{eq.PFODE} for simulating the flow.
After positive-time regularization and under suitable drift–score regularity, this ODE generates a {\it confined} regular Lagrangian flow and transports the same marginals as the reflected diffusion, even though it contains no explicit reflection term. This supports deterministic sampling between positive times and clarifies that the relevant boundary condition is tangency of the full probability-flow velocity, which does not necessarily require a zero normal score at the boundary. 
Importantly, the result is not restricted to hypercubes or polytopes; Corollary \ref{cor:early} applies to general bounded $C^{2,\alpha}$
 domains such as the ones considered by Fishman et al. \cite{fishman2024}.
 
The results also identify important limitations of the deterministic sampling method. Equality of the marginal distributions alone does not guarantee a well-posed or stable deterministic sampler \eqref{eq.PFODE}: singular initial data can obstruct a flow at time zero, and boundary degeneracy can produce unbounded compression even when confined characteristics and exact marginal transport exist. Thus early stopping, control of the combined drift–score field, and correct boundary behavior are genuine mathematical requirements rather than technical afterthoughts.

These effects illustrate the gap between Eulerian and Lagrangian well-posedness, which can be relevant in applications. Proposition \ref{prop:energyuniq} shows that the no-flux Fokker–Planck density evolution is unique for bounded measurable forward drifts, whereas construction of a regular Lagrangian flow for the probability-flow ODE requires additional regularity, divergence, and boundary-trace assumptions on the combined drift–score field. Thus the density flow and  its time reversal are uniquely determined under weaker assumptions than those needed to justify deterministic probability-flow sampling
via \eqref{eq.PFODE}.

From a theoretical perspective, many ramifications of the problems studied here remain to be explored.
A natural extension would be  to explore the case of obliquely reflected diffusions.
The Eulerian object is then, in general, a measure pair --interior
measure and boundary measure \cite{HarrisonWilliams,KangRamananStat}-- and the
tangential component of the reflection field induces a surface flux
along $\partial\Omega$, so that tangency of the probability-flow
velocity fails at the level of the boundary condition itself.
Whether a deterministic probability flow exists
in this setting is not clear. In this regard, the compression analysis of Theorem
\ref{thm:B} provides a natural model for the interaction of a
tangential wall flux with Lagrangian transport.

\subsection*{Acknowledgements} The author thanks Luigi Ambrosio for his beautiful lectures at the Spring School on Mathematics of Random Systems (Pisa 2025)  and RenYuan Xu for enlightening discussions on generative models. 

%=====================================================================

%=====================================================================

\end{document}